\documentclass[11pt,reqno]{amsart}

\usepackage[usenames]{color}
\usepackage{graphicx}
\usepackage{amscd}
\usepackage[colorlinks=true,
linkcolor=webgreen,
filecolor=webbrown,
citecolor=webgreen]{hyperref}

\definecolor{webgreen}{rgb}{0,.5,0}
\definecolor{webbrown}{rgb}{.6,0,0}
\usepackage{amsthm}
\usepackage{fullpage}
\usepackage{enumerate}
\usepackage{epstopdf}
\usepackage{graphics,amsmath,amssymb}
\usepackage{graphicx}

\usepackage{amsfonts}
\usepackage{latexsym}
\usepackage{url}

\theoremstyle{plain}
\newtheorem{theorem}{Theorem}[section]
\newtheorem{lemma}[theorem]{Lemma}
\newtheorem{corollary}[theorem]{Corollary}
\newtheorem{proposition}[theorem]{Proposition}

\theoremstyle{definition}

\theoremstyle{remark}

\def\adots{
  \mathinner{\mkern1mu\raise1pt\hbox{.}\mkern2mu\raise4pt\hbox{.}
  \mkern2mu\raise7pt\vbox{\kern7pt\hbox{.}}\mkern1mu}}
\newcommand \lcm{\operatorname{lcm}}

\begin{document}

\begin{center}

\title[Characterization of the strong divisibility property for generalized Fibonacci polynomials]
 {Characterization of the strong divisibility property for generalized Fibonacci polynomials}
 \author{Rigoberto Fl\'orez}
 \address{Department of Mathematical Sciences\\
		The Citadel\\
		Charleston, SC \\
		U.S.A.}
  \email{rflorez1@citadel.edu}
  \author{Robinson A. Higuita}
 \address{Instituto de Matem\'aticas,\\
        Universidad de Antioquia,\\
        Medell\'in, Colombia}
 \email{robinson.higuita@udea.edu.co}
 \author{Antara Mukherjee}
 \address{Department of Mathematical Sciences\\
		The Citadel\\
		Charleston, SC \\
		U.S.A.}
 \email{mukherjeea1@citadel.edu}
\end{center}

\maketitle
\centerline{\bf Abstract}
\noindent
 It is known that the greatest common divisor of two Fibonacci numbers is again a Fibonacci number.
 This is called the  \emph{strong divisibility property}. However, strong divisibility does not hold for
 every second order sequence. In this paper we study the generalized Fibonacci polynomials and
 classify them in two types depending on their Binet formula. We give a complete
 characterization of those polynomials that  satisfy the strong divisibility property.  We also give
 formulas to calculate the greatest common divisor of those polynomials that do not satisfy the strong divisibility property.

 Note. This paper is now published in INTEGERS.

\section {Introduction}

It is well-known that the greatest common divisor ($\gcd$) of two Fibonacci numbers is a Fibonacci
number \cite{koshy}. In fact, $\gcd(F_m,F_n)=F_{\gcd(m,n)}$.  This is called  the
\emph{strong divisibility property} or \emph{Fibonacci gcd property}. We study divisibility
properties of generalized Fibonacci polynomials (GFP) and in particular we give a characterization
of the  strong divisibility property for these polynomials.

We classify the GFPs into two types, the Lucas type and the Fibonacci type, depending on their closed formulas
or their Binet formulas (see for example, $L_n(x)$ (\ref{bineformulados}) and $R_n(x)$ (\ref{bineformulauno}),
and Table \ref{equivalent}). That is, if after solving the characteristic polynomial of a GFP sequence we obtain a
closed formula that looks like the Binet formula for Fibonacci (Lucas) numbers, we call the sequence a Fibonacci (Lucas)
type sequence. Familiar examples of Fibonacci type polynomials are:  Fibonacci polynomials,
 Pell polynomials, Fermat polynomials, Chebyshev polynomials of the second kind, Jacobsthal polynomials,
 and one type of Morgan-Voyce polynomials.  Examples of Lucas type polynomials are:
 Lucas polynomials, Pell-Lucas polynomials, Fermat-Lucas polynomials, Chebyshev polynomials of the first kind,
 Jacobsthal-Lucas polynomials, and the second type of Morgan-Voyce  polynomials. Horadam \cite{horadam-synthesis} and
 Andr\'e-Jeannin \cite{Richard} have studied these polynomials in detail.

In Theorem \ref{gcd:property:fibonacci} we prove that a GFP satisfies the strong divisibility property if
and only if it is of Fibonacci type. Theorem \ref{second:main:thm} shows that the Lucas type polynomials
partially satisfy the strong divisibility property and also gives the $\gcd$  for those cases in which the
property is not satisfied.

A Lucas type polynomial is equivalent (or conjugate) to a Fibonacci type polynomial if they both have the same recurrence
relation but different initial conditions (see also Fl\'orez et al. \cite{florezHiguitaMuk}).
Theorem \ref{combine:gcd:Lucas:Fibobacci} proves that two equivalent GFPs  partially satisfy
the strong divisibility property  and  gives the $\gcd$  for those cases in which the property is
not satisfied.

In 1969 Webb and Parberry \cite{Webb} extended the strong divisibility property to Fibonacci polynomials. In 1974 Hoggatt and
Long \cite{HoggattLong} proved the strong divisibility property for one type of  generalized Fibonacci polynomial. In 1978 Hoggatt
and Bicknell-Johnson \cite{hoggatt} extended the result mentioned in \cite{HoggattLong} to some cases of Fibonacci type polynomials.
In 2005 Rayes, et al. \cite{Rayes} proved that the strong divisibility property holds partially for the Chebyshev polynomials
(we prove the general result in Theorem \ref{second:main:thm}). Over the years several other authors
\cite{Barbero, bliss, hall, kimberling, kimberlingStrong1, kimberlingStrong2, lucas, mcdaniel, norfleet, nowicki, ward}
have also been interested in the divisibility properties of sequences.

Lucas \cite{lucas} proved the strong divisibility property (SDP) for Fibonacci numbers. However, the study of the
SDP for Lucas numbers did not occur until 1991, when McDaniel \cite{mcdaniel} proved that the Lucas numbers  partially satisfy
the SDP. In 1995 Hilton, et al. \cite{hilton} gave some more precise results about this property.
As mentioned above, several authors have been interested in the divisibility properties for
Fibonacci type polynomials. However, the Lucas type polynomials have been less studied. Here we give a complete study of all
three cases of the SDP. Indeed, we give a characterization of the SDP for Fibonacci type polynomials and
study both the SDP for Lucas type polynomials and the SDP for the combinations of Lucas type polynomials and
Fibonacci type polynomials. Finally we provide an open question for the most general case of  the combination of
two polynomials.

\section{Generalized Fibonacci polynomials}\label{General:Fibonacci:Polynomial}
In the literature there are several definitions of generalized Fibonacci polynomials.
The definition that we introduce here is simpler and covers other definitions.
The background given in this section is a summary of the background given in \cite{florezHiguitaMuk}.
However, the definition of generalized Fibonacci polynomial here is not exactly the same as in \cite{florezHiguitaMuk}.
The \emph{generalized Fibonacci polynomial} sequence $\{ G_{n}(x)\}$, denoted by GFP, is defined by the following recurrence relation:
if $p_0(x)$ is a constant and $p_1(x)$, $d(x)$, and $g(x)$ are non-zero polynomials in $\mathbb{Z}[x]$ with $\gcd(d(x), g(x))=1$, then
\begin{equation}\label{Fibonacci;general}
G_0(x)=p_0(x), \; G_1(x)= p_1(x),\;  \text{and} \;  G_{n}(x)= d(x) G_{n - 1}(x) + g(x) G_{n - 2}(x)
\end{equation}
  for  $n\ge 2$.

For example, if we let $p_0(x)=0$, $p_1(x)=1$, $d(x)=x$, and $g(x)=1$ we obtain the regular Fibonacci polynomial sequence. Thus,
\[F_0(x)= 0, \; F_1(x)= 1, \; \text{and} \;  F_{n}(x)= x F_{n - 1}(x) + F_{n - 2}(x) \text{ for } n\ge 2.\]

Letting $x=1$ and choosing the correct values for $p_0(x)$, $p_1(x)$, $d(x)$, and $g(x)$,  the generalized Fibonacci
polynomial sequence gives rise to three classical numerical sequences: the Fibonacci sequence, the Lucas sequence and the
generalized Fibonacci sequence.

Table \ref{familiarfibonacci} provides familiar examples of the GFPs  (see \cite{florezHiguitaMuk, Pell,Fermat,koshy}).
Schechter polynomials \cite{hoggatt} are also examples of generalized Fibonacci polynomials.

\begin{table} [!ht]
\begin{center}\scalebox{0.8}{
\begin{tabular}{|l|l|l|l|} \hline
  Polynomial             & Initial value     & Initial value  & Recursive Formula 						                \\	
    			         &$G_0(x)=p_0(x)$  	 &$G_1(x)=p_1(x)$ &$G_{n}(x)= d(x) G_{n - 1}(x) + g(x) G_{n - 2}(x)$ 	    \\  \hline \hline
  Fibonacci              & $0$	             &$1$             &$F_{n}(x) = x F_{n - 1}(x) + F_{n - 2}(x)$	 	        \\
  Lucas 	             &$2$	             &$x$ 	          &$D_n(x)= x D_{n - 1}(x) + D_{n - 2}(x)$                  \\ 						
  Pell			    	 &$0$	             &$1$             &$P_n(x)= 2x P_{n - 1}(x) + P_{n - 2}(x)$                 \\
  Pell-Lucas 	    	 &$2$	             &$2x$            &$Q_n(x)= 2x Q_{n - 1}(x) + Q_{n - 2}(x)$                 \\
  Pell-Lucas-prime 	     &$1$	             &$x$             &$Q_n^{'}(x)= 2x Q_{n - 1}^{'}(x) + Q_{n - 2}^{'}(x)$     \\
  Fermat  	             &$0$	             &$1$             &$\Phi_n(x)= 3x\Phi_{n-1}(x)-2\Phi_{n-2}(x) $             \\
  Fermat-Lucas	         &$2$	             &$3x$  	      &$\vartheta_n(x)=3x\vartheta_{n-1}(x)-2\vartheta_{n-2}(x)$\\
  Chebyshev second kind  &$0$	             &$1$             &$U_n(x)= 2x U_{n-1}(x)-U_{n-2}(x)$  	 	                \\
  Chebyshev the first kind   &$1$	             &$x$             &$T_n(x)= 2x T_{n-1}(x)-T_{n-2}(x)$                       \\	 	
  Jacobsthal  	  		 &$0$	             &$1$             &$J_n(x)= J_{n-1}(x)+2xJ_{n-2}(x)$ 	 	                \\
  Jacobsthal-Lucas	     &$2$		         &$1$             &$j_n(x)=	j_{n-1}(x)+2xj_{n-2}(x)$                        \\
  Morgan-Voyce	         &$0$		         &$1$             &$B_n(x)= (x+2) B_{n-1}(x)-B_{n-2}(x) $  	 	            \\
  Morgan-Voyce 	         &$2$		         &$x+2$           &$C_n(x)= (x+2) C_{n-1}(x)-C_{n-2}(x)$  	 	            \\
  Vieta 		         &$0$ 	   	         &$1$	          &$V_n(x)=x V_{n-1}(x)-V_{n-2}(x)$ 	                    \\
  Vieta-Lucas 		     &$2$ 	   	         &$x$	          &$v_n(x)=x v_{n-1}(x)-v_{n-2}(x)$                         \\
  \hline
\end{tabular}}
\end{center}
\caption{Recurrence relation of some GFPs.} \label{familiarfibonacci}
\end{table}

\subsection{Fibonacci type and Lucas type polynomials}

If we impose some conditions on Definition (\ref{Fibonacci;general}) we obtain two types of distinguishable polynomials.
We say that a sequence as in  (\ref{Fibonacci;general}) is Lucas type or the first type,  if
$2p_{1}(x)=p_{0}(x)d(x)$ with  $p_{0}\ne 0$.  We say that a sequences as in
(\ref{Fibonacci;general}) is Fibonacci type or the second type if $p_{0}(x)=0$ and  $p_{1}(x)$ a non-zero constant.

If $d^2(x)+4g(x)> 0$, then the explicit formula for the recurrence relation (\ref{Fibonacci;general}) is given by

\begin{equation}\label{solutionrecurrencerelationuno}
 G_{n}(x) = t_1 a^{n}(x) + t_2 b^{n}(x)
\end{equation}
where $a(x)$ and $b(x)$ are the solutions of the quadratic equation associated with the second order
recurrence relation $G_{n}(x)$. That is,  $a(x)$ and $b(x)$ are the solutions of $z^2-d(x)z-g(x)=0$.
The explicit formula for $G_{n}(x)$ given in (\ref{solutionrecurrencerelationuno})
with $G_{0}(x)=p_{0}(x)$ and $G_{1}(x)=p_{1}(x)$ implies that
\begin{equation}\label{solutionrecurrencerelationdos}
 t_{1}=\dfrac{p_{1}(x)-p_{0}(x)b(x)}{a(x)-b(x)} \quad \text{ and } \quad  t_{2}=\dfrac{-p_{1}(x)+p_{0}(x)a(x)}{a(x)-b(x)}.
\end{equation}

Using (\ref{solutionrecurrencerelationuno}) and  (\ref{solutionrecurrencerelationdos}) we obtain the Binet
formulas for the generalized Fibonacci sequences of Lucas type and  Fibonacci type. Thus, substituting
$2p_{1}(x)=p_{0}(x)d(x)$ in (\ref{solutionrecurrencerelationdos}) we obtain that $t_{1}=t_{2}= p_{0}(x)/2$.
Substituting these values of $t_{1}$ and $t_{2}$ in (\ref{solutionrecurrencerelationuno}) and letting  $\alpha$ be  $2/p_{0}(x)$  we obtain:
the Binet formula for generalized Fibonacci sequence of Lucas type,
\begin{equation}\label{bineformulados}
L_n(x)=\frac{a^{n}(x)+b^{n}(x)}{\alpha}.
\end{equation}
We want $\alpha$ to be an integer, and therefore  $|p_{0}(x)|=1 \text { or } 2$.

Now, substituting $p_{0}(x)=0$ and the constant $p_{1}(x)$ in (\ref{solutionrecurrencerelationdos})
we obtain that $t_{1}=t_{2}=p_{1}(x)$.
Substituting this in (\ref{solutionrecurrencerelationuno})  we obtain  the Binet formula for the
generalized Fibonacci sequence of Fibonacci type,
\begin{equation}\label{bineformulaunogeneral}
R_n(x)=\frac{p_{1}(x)\left(a^{n}(x)-b^{n}(x)\right)}{a(x)-b(x)}.
\end{equation}
In this paper we are interested only in $R_n(x)$ when $p_1(x)=1$ (see also \cite{HoggattLong}). Therefore, the Binet formula $R_n(x)$ used here is:
\begin{equation}\label{bineformulauno}
R_n(x)=\frac{a^{n}(x)-b^{n}(x)}{a(x)-b(x)}.
\end{equation}
Note that if $d(x)$ and $g(x)$ are the polynomials defined in \eqref{Fibonacci;general}, then
 $a(x)+b(x)=d(x)$, $a(x)b(x)= -g(x)$, and $a(x)-b(x)=\sqrt{d^2(x)+4g(x)}$ .

The sequence of polynomials that have Binet representations $R_n(x)$ or $L_n(x)$ depend only on $d(x)$ and $g(x)$ defined
in  \eqref{Fibonacci;general}.  We say that a generalized Fibonacci sequence of Lucas (Fibonacci) type is \emph{equivalent}, or
the \emph{conjugate}, to a sequence of the Fibonacci (Lucas) type, if their recursive sequences are determined by the same
polynomials $d(x)$ and $g(x)$. Notice that two equivalent polynomials have the same $a(x)$ and $b(x)$ in their Binet
representations.

For example, the Lucas polynomial is a GFP of Lucas type, whereas the Fibonacci polynomial is a GFP of
Fibonacci type. Lucas and Fibonacci
polynomials are equivalent  because   $d(x)=x$ and $g(x)=1$ for both (see Table \ref{familiarfibonacci}). Note that in their Binet
representations they both have  $a(x)= (x+\sqrt{x^2+4})/2$ and $b(x)=(x-\sqrt{x^2+4})/2$.
Table \ref{equivalent} is based on information from the following papers \cite{Richard, florezHiguitaMuk, horadam-synthesis}.
The leftmost polynomials in Table \ref{equivalent} are of the Lucas type and their equivalent polynomials
are in the second column on the same line. In the last two columns of Table \ref{equivalent} are the $a(x)$ and $b(x)$ that
the pairs of equivalent polynomials share. It is easy to obtain the characteristic equations of the sequences given in
Table \ref{familiarfibonacci} where the roots of each equation are given by  $a(x)$ and $b(x)$.

For the sake of simplicity throughout this paper we use $a$ in place of $a(x)$ and $b$ in place of $b(x)$ when they appear in the
Binet formulas. We use the notation $G_n^{*}$ or $G_n^{\prime}$ for $G_n$ depending on whether it satisfies the Binet formulas
 (\ref{bineformulados}) or  (\ref{bineformulauno}), respectively, (see Section \ref{mainresults:section}).

For most of the proofs of GFPs of Lucas type it is required that
$$ \gcd(p_0(x), p_1(x))=1, \quad \gcd(p_0(x), d(x))=1, \quad \text{ and } \quad \gcd(d(x), g(x))=1.$$
It is easy to see that $\gcd(\alpha, G_{n}^{*}(x))=1$.
Therefore, for the rest of the paper we suppose that these conditions mentioned hold for all
GFP sequences of Lucas type treated here.
We use $\rho$ to denote $\gcd(d(x),G_1(x))$. Notice that in the definition of Pell-Lucas we have
$p_0(x)=2$ and $p_1(x)=2x$. Thus, the $\gcd(p_0(x),p_1(x))\ne 1$.
Therefore, Pell-Lucas does not satisfy the extra conditions that were just imposed on the GFPs.
To solve this problem we define \emph{Pell-Lucas-prime} as follows:
\[Q_0^{'}(x)= 1, \; Q_1^{'}(x)= x, \; \text{and} \;  Q_{n}^{'}(x)= 2x Q_{n - 1}^{'}(x) + Q_{n - 2}^{'}(x) \text{ for } n\ge 2.\]
It easy to see  that $2Q_{n}^{'}(x)=Q_{n}(x)$ and that $\alpha=2$.
Fl\'orez, et al.  \cite{florezHiguitaJunes} worked on similar problems for numerical sequences.

\begin{table} [!ht]
\begin{center}
\scalebox{0.8}{
\begin{tabular}{|l|l|l|l|} \hline
   Polynomial  	    	  	&Polynomial of 		&$a(x)$ 	               & $b(x)$			         \\	
   Lucas type 	 	        &Fibonacci type  	& 					       &      				     \\ \hline \hline
   Lucas 			       	&Fibonacci 		    &  $(x+\sqrt{x^2+4})/2$    & $(x-\sqrt{x^2+4})/2$    \\ 						
   Pell-Lucas-prime 	   	&Pell			    &  $x+\sqrt{x^2+1}$	       & $x-\sqrt{x^2+1}$        \\
   Fermat-Lucas 	       	&Fermat 			&  $(3x+\sqrt{9x^2-8})/2$  & $(3x-\sqrt{9x^2-8})/ 2$ \\
   Chebyshev 1st kind       &Chebyshev 2nd kind &  $x +\sqrt{x^2-1}$       & $x -\sqrt{x^2-1}$       \\
   Jacobsthal-Lucas	       	&Jacobsthal  		&  $(1+\sqrt{1+8x})/2$     & $(1-\sqrt{1+8x})/2$     \\
   Morgan-Voyce 	        &Morgan-Voyce	    &  $(x+2+\sqrt{x^2+4x})/2$ & $(x+2-\sqrt{x^2+4x})/2$ \\
   Vieta-Lucas              &Vieta            	& $(x+\sqrt{x^2-4})/2$     & $(x-\sqrt{x^2-4})/2$    \\ \hline
\end{tabular}}
\end{center}
\caption{$R_n(x)$ equivalent to $L_n(x)$.} \label{equivalent}
\end{table}

\textbf{Note.}  The definition of GFPs in \cite{florezHiguitaMuk}
differs from the definition in this paper due to the initial conditions of the Fibonacci type polynomials.
Thus, the initial condition for the Fibonacci type polynomials in \cite{florezHiguitaMuk} is  $G_{0}(x)=p_{0}(x)=1$
and so implicitly  $G_{-1}(x)=0$. However, our definition for the Lucas type polynomials is the same in both papers.

\section{Divisibility properties of GFPs}

In this section we prove a few divisibility and $\gcd$ properties which are true for all GFPs. These
results will be used in a  later section to prove the main results of this paper.

Proposition \ref{prop2;1} parts (1) and (2) is a generalization of Proposition 2.2 in  \cite{florezjunes}.
The proof here is similar to the proof in \cite{florezjunes} since both use properties of integral domains.
The reader can therefore update the proof in the
afore-mentioned paper to obtain the proof of this proposition.

\begin{proposition}\label{prop2;1}
Let $p(x), q(x), r(x),$ and $s(x)$ be polynomials.
\begin{enumerate}[(1)]
  \item If $\gcd(p(x),q(x))=\gcd(r(x),s(x))=1$, then $\gcd(p(x)q(x),r(x)s(x))$\\
  		is equal to $\gcd(p(x),r(x))\gcd(p(x),s(x))\gcd(q(x),r(x))\gcd(q(x),s(x)).$

  \item If $\gcd(p(x),r(x))=1$ and $\gcd(q(x),s(x))=1$, then
     	$$\gcd(p(x)q(x),r(x)s(x))=\gcd(p(x),s(x))\gcd(q(x),r(x)).$$
  \item If $z_1(x)=\gcd(p(x),r(x))$ and $z_2(x)=\gcd(q(x),s(x))$, then
     	$$\gcd(p(x)q(x),r(x)s(x))=\frac{\gcd( z_2(x)p(x),z_1(x)s(x))\gcd( z_1(x)q(x), z_2(x)r(x))}{z_1(x)z_2(x)}.$$
\end{enumerate}
\end{proposition}

\begin{proof} We prove part (3). Since $\gcd(p(x),r(x))=z_1(x) $ and $\gcd(q(x),s(x))=z_2(x)$,
	there are polynomials $P(x)$, $S(x)$, $Q(x)$, and
	$R(x)$ with $$\gcd(P(x),R(x))=\gcd(Q(x),S(x))=1,$$
 such that $ p(x) = z_1(x)  P(x)$,\;  $s(x)= z_2(x)  S(x)$,\;  $r(x) = z_1(x)  R(x)$,  and
$q(x)= z_2(x)  Q(x)$. So,
\begin{eqnarray*}
\gcd(p(x)q(x),r(x)s(x)) &=&\gcd(z_1(x) P(x) z_2(x)Q(x), z_1(x) R(x) z_2(x)S(x))\\
&=&z_1(x)z_2(x)\gcd( P(x)Q(x), R(x)S(x)).
\end{eqnarray*}
From part (2) we know that
$$
\gcd( P(x)Q(x), R(x) S(x))=\gcd( P(x),S(x))\gcd( Q(x), R(x)).
$$
Now it is easy to see that
\[
\gcd(p(x)q(x),r(x)s(x))=\frac{\gcd( z_2(x)p(x),z_1(x)s(x))\gcd( z_1(x)q(x), z_2(x)r(x))}{z_1(x)z_2(x)}.
\]
This proves part (3).
\end{proof}

We recall that  $\rho=\gcd(d(x),G_1(x))$ and that for GFPs of Lucas type it is required that
$\gcd(p_0(x), p_1(x))=1$, $\gcd(p_0(x), d(x))=1$,  $\gcd(p_0(x), g(x))=1$, and that $\gcd(d(x), g(x))=1$.
We also recall that $p_0(x)=0$ and $p_1(x)=1$ for GFPs of Fibonacci type.

For the rest of the paper we use the notation $G_n^{*}$ if the GFP $G_n$ satisfies the Binet formula
(\ref{bineformulados}) and $G_n^{\prime}$ if the GFPs $G_n$ satisfies the Binet formula
(\ref{bineformulauno}). We use  $G_n$ if the result
does not need the mentioned classification to be true. We recall that for Lucas type polynomials
$|p_{0}(x)| =1 \text{ or } 2$ and for Fibonacci type polynomials $p_{1}(x) =1$.   Lemma \ref{gcdlemmas} part (3) is \cite[Lemma 3]{HoggattLong}.

\begin{lemma} \label{gcdlemmas} If $G_{n}(x)$ is a GFP of either Lucas or Fibonacci type, then
\begin{enumerate}[(1)]
   \item $\gcd(d(x), G_{2n+1}(x))=G_1(x)$ for every positive integer $n$.

    \item If the GFP is of Lucas type, then $\gcd(d(x), G_{2n}^{*}(x))= 1$ and

          if the GFP is of Fibonacci type, then $\gcd(d(x), G_{2n}^{\prime}(x))= d(x)$.

    \item If $n$ is a positive integer, then $\gcd(g(x), G_n(x))=\gcd(g(x), G_{1}(x))=1$.
\end{enumerate}
\end{lemma}

\begin{proof}
We prove part (1) by induction. Let $G_n$ be a GFP and  let $S(n)$ be the statement
\[\rho=\gcd(d(x), G_{2n+1}(x))\text{  for  }n\ge 1.\]
To prove $S(1)$ we suppose that $\gcd(d(x), G_{3}(x))=r$.
Thus, $r$ divides any linear combination of $d(x)$ and $G_3(x)$. Therefore, $r$ divides $G_3(x)-d(x) G_2(x)$. This and given that
$G_3(x)= d(x) G_{2}(x) + g(x) G_{1}(x)$ imply that $r\mid g(x)G_{1}(x)$. So, $r\mid \gcd(d(x),g(x)G_{1}(x))$. Since
$\gcd(d(x), g(x))=1$, we have $r\mid \rho$. It is easy to see that $\rho \mid r$. Thus, $r=\gcd(d(x),G_1(x))$. This proves $S(1)$.

We suppose that $S(n)$ is true for $n=k-1$. That is, $\gcd(d(x), G_{2k-1}(x)) = \rho$. To prove $S(k)$ we suppose that
$\gcd(d(x), G_{2k+1}(x))=r'$.  Thus, $r'$ divides any linear combination of $d(x)$ and $G_{2k+1}(x)$. Therefore,
$r'\mid \left(G_{2k+1}(x)-d(x) G_{2k}(x)\right)$. This and  $G_{2k+1}(x)= d(x) G_{2k}(x) + g(x) G_{2k-1}(x)$ imply that
$r'\mid g(x)G_{2k-1}(x)$. Therefore, $r'\mid \gcd(d(x),g(x)G_{2k-1}(x))$. Since \\
\noindent $\gcd(d(x), g(x))=1$, we have that
$r'$ divides the $\gcd(d(x),G_{2k-1}(x))$. By the inductive hypothesis we know that $\gcd(d(x), G_{2k-1}(x)) = \rho$. Thus,
$r'\mid \rho$.  It is easy to see that $\gcd(d(x),G_{2k+1}(x))$ divides $r'$. So, $r'=\gcd(G_1(x), d(x))$.

We show that depending on the type of sequence, it holds that $\gcd(d(x),G_1(x))$ is equal to $G_1$.
If $G_n(x)$ is a GFP of Fibonacci type, then by definition of $p(x)$ we have $G_1(x)=1$
(see comments after Binet formula \eqref{bineformulaunogeneral}).
Suppose that  $G_n(x)$ is a GFP of Lucas type. Recall that $2p_{1}(x)=p_{0}(x)d(x)$ and that $|p_{0}(x)|=1 \text { or } 2$.
The conclusion is straightforward since $G_1(x)=(a(x)+b(x))/\alpha=d(x)/\alpha$.

Proof of part (2). Let $S(n)$ be the statement
\[\rho=\gcd(d(x), G_{2n}(x))\text{  for  }n\ge 1.\]
To prove $S(2)$ we suppose that $\gcd(d(x), G_{4}(x))=r$.
Thus, $r$ divides any linear combination of $d(x)$ and $G_4(x)$. Therefore, $r$ divides $G_4(x)-d(x) G_3(x)$. This and given that
$G_4(x)= d(x) G_{3}(x) + g(x) G_{2}(x)$, imply that $r\mid g(x)G_{2}(x)$. Therefore, $r\mid \gcd(d(x),g(x)G_{2}(x))$. Since
$\gcd(d(x), g(x))=1$, we have $r\mid \rho$. It is easy to see that $\rho \mid r$. Thus, $r=\gcd(d(x),G_2(x))$. This proves $S(2)$.

We suppose that $S(n)$ is true for $n=k-1$. That is, $\gcd(d(x), G_{2k-2}(x)) = \rho$. To prove $S(k)$ we suppose that
$\gcd(d(x), G_{2k}(x))=r'$.  Thus, $r'$ divides any linear combination of $d(x)$ and $G_{2k}(x)$. So,
$r' \mid \left(G_{2k}(x)-d(x) G_{2k-1}(x)\right)$. This and $G_{2k}(x)= d(x) G_{2k-1}(x) + g(x) G_{2k-2}(x)$ imply that
$r'\mid g(x)G_{2k-2}(x)$. Therefore, $r'\mid \gcd(d(x),g(x)G_{2k-2}(x))$. Since $\gcd(d(x), g(x))=1$, we have that
$r'$ divides  the $\gcd(d(x),G_{2k-2}(x))$. From the inductive hypothesis we know that $\gcd(d(x), G_{2k-2}(x)) = \rho$. Thus,
$r'\mid \rho$.  It is easy to see that $\gcd(d(x),G_{2k}(x))$ divides $r'$. Therefore, $r'=\gcd(G_2(x), d(x))$.

We observe that for a GFP of Fibonacci type it holds that $G_2^{\prime}(x)=a(x)+b(x)=d(x)$. So, it is clear that
$\gcd(G_{2n}^{\prime}(x),d(x))=d(x)$. For a GFP of Lucas type it holds that $G_0^{*}(x)$ is a non-zero constant.
Since $G_2^{*}(x)=d(x)G_1^{*}(x)+g(x)G_{0}^{*}(x)$, and $\gcd(d(x), g(x))=1$, we have
$$\gcd(d(x),G_2^{*}(x))=\gcd(d(x),d(x)G_1^{*}(x)+g(x)G_{0}^{*}(x))=\gcd(d(x),g(x)G_0^{*}(x))=1.$$

Proof of part (3). We prove that  $\gcd(g(x), G_1(x))=1$ by cases. If $G_1(x)$ is of the Fibonacci type,
the conclusion is straightforward.
As a second case we suppose that $G_1(x)$ is of the Lucas type. That is, $G_1(x)$
satisfies the Binet formula (\ref{bineformulados}). Therefore, we have
\[\gcd(g(x), G_1(x))=\gcd(g(x), L_1(x)) = \gcd(g(x), [a+b]/\alpha)=\gcd(g(x), d(x)/\alpha).\]
Since $\gcd(g(x), d(x))=1$, we have  $\gcd(g(x), d(x)/\alpha)=1$.
This completes the proof.
\end{proof}

\begin{lemma} \label{gcdforproposition4}  If $\{G_{n}(x)\}$ is a GFP sequence,
then for every positive integer $n$ the following holds:
\begin{enumerate}[(1)]
   \item $\gcd(G_{n}(x), G_{n+1}(x))$ divides $\gcd(G_{n}(x),g(x)G_{n-1}(x))=\gcd(G_{n}(x),G_{n-1}(x))$,
   \item $\gcd(G_{n}(x), G_{n+2}(x))$ divides $\gcd(G_{n}(x), d(x)G_{n+1}(x))$.
\end{enumerate}
\end{lemma}
\begin{proof} We prove part (1), the proof of part (2) is similar and we omit it.  If $r$ is equal to $ \gcd(G_{n}(x), G_{n+1}(x))$, then  $r$
divides any linear combination of $G_{n}(x)$ and $G_{n+1}(x)$. Therefore, $r\mid (G_{n+1}(x)-d(x) G_{n}(x))$. This and the recursive
definition of $G_{n+1}(x)$ imply that $r\mid g(x)G_{n-1}(x)$. Therefore, $r\mid \gcd(g(x)G_{n-1}(x),G_{n}(x))$. Since $\gcd(g(x),G_{n}(x))=1$, we have  $$\gcd(g(x)G_{n-1}(x),G_{n}(x))=\gcd(G_{n-1}(x),G_{n}(x)).$$
This completes the proof.
\end{proof}

Note that Proposition \ref{gcddistance1;2} part (2) when $m=n+1$ is \cite[Theorem 4]{HoggattLong}.

\begin{proposition} \label{gcddistance1;2} Let $m$ and $n$ be positive integers with $0<|m-n|\le 2$.

\begin{enumerate}[(1)]
 \item If $G_{t}^{*}(x)$  is a  GFP of Lucas type, then
 \[
 \gcd(G_m^{*}(x),G_n^{*}(x))=
 \begin{cases}
         G_{1}^{*}(x), & \mbox{if $m$ and $n$ are both odd;} \\
         1, & \mbox{otherwise. }
\end{cases}
 \]

   \item If $G_{t}^{\prime}(x)$  is a  GFP of Fibonacci type, then
 \[
 \gcd(G_m^{\prime}(x),G_n^{\prime}(x))=
 \begin{cases}
        G_{2}^{\prime}(x), & \mbox{if $m$ and $n$ are both even;} \\
         1, & \mbox{otherwise. }
\end{cases}
 \]

\end{enumerate}

\end{proposition}

\begin{proof}   We prove part (1) using several cases based on the values of $m$ and $n$. The proof of part (2) is similar and we omit it.
We first provide the proof for the
case when $m$ and $n$ are consecutive integers using induction on $m$. Let $S(m)$ be the statement
\[\gcd(G_{m}^{*}(x), G_{m+1}^{*}(x))=1\text{  for } m\ge 1.\]
We prove $S(1)$. From Lemma \ref{gcdforproposition4} part (1) we know that
\begin{equation}\label{formula1:for:gcddistance1;2}
\gcd(G_{1}^{*}(x),G_{2}^{*}(x)) \text{ divides } \gcd(G_{1}^{*}(x), g(x)G_{0}^{*}(x)).
\end{equation}
Since
$$\gcd(G_{0}^{*}(x),G_{1}^{*}(x))=\gcd(p_{0}(x),p_{1}(x))=1,$$
we have
$$\gcd(G_{1}^{*}(x),g(x)G_{0}^{*}(x))=\gcd(G_{1}^{*}(x),g(x)).$$
This, (\ref{formula1:for:gcddistance1;2}), and Lemma \ref{gcdlemmas} part (3) imply that $\gcd(G_{1}^{*}(x),G_{2}^{*}(x))=1$.

We suppose that $S(m)$ is true for $m=k-1$. Thus, $\gcd(G_{k-1}^{*}(x),G_{k}^{*}(x))=1$. We prove that $S(k)$ is true.
From Lemma \ref{gcdforproposition4} part (1) we know that
\begin{equation}\label{formula2:for:gcddistance1;2}
\gcd(G_{k}^{*}(x),G_{k+1}^{*}(x)) \text{ divides } \gcd(G_{k}^{*}(x), g(x)G_{k-1}^{*}(x)).
\end{equation}
From Lemma \ref{gcdlemmas} part (3) we know that $\gcd(G_{k}^{*}(x),g(x))=1$.
Therefore,
$$\gcd(G_{k}^{*}(x), g(x)G_{k-1}^{*}(x))= \gcd(G_{k}^{*}(x), G_{k-1}^{*}(x)).$$ This, (\ref{formula2:for:gcddistance1;2}),
 and the inductive hypothesis imply that $\gcd(G_{k}^{*}(x),G_{k+1}^{*}(x)=1$.

We prove the proposition for consecutive even integers (this proof is actually
a direct consequence of the previous proof).
From Lemma \ref{gcdforproposition4} part (2), we have that
$\gcd(G_{2k}^{*}(x), G_{2k+2}^{*}(x))$ divides $\gcd(G_{2k}^{*}(x),d(x)G_{2k+1}^{*}(x))$.
From Lemma \ref{gcdlemmas} part (2) we know that $\gcd(d(x),G_{2k}^{*}(x))=1$.
This implies that
$$\gcd(G_{2k}^{*}(x),d(x)G_{2k+1}^{*}(x))= \gcd(G_{2k}^{*}(x),G_{2k+1}^{*}(x)).$$
From the previous part  of this proof --that is, the case when $m$ and $n$ are
consecutive integers-- we conclude that $\gcd(G_{2k}^{*}(x),G_{2k+1}^{*}(x))=1$.
This proves that
$$\gcd(G_{2k}^{*}(x), G_{2k+2}^{*}(x))=1.$$

Finally we prove the proposition for consecutive odd integers. From the recursive definition of GFPs we have
$\gcd(G_{2k+1}^{*}(x),G_{2k-1}^{*}(x))$ is equal to
$$
\gcd(d(x)G_{2k}^{*}(x)+g(x)G_{2k-1}^{*}(x),G_{2k-1}^{*}(x))=\gcd(d(x)G_{2k}^{*}(x),G_{2k-1}^{*}(x)).
$$
From the first case in this proof we know that   $\gcd(G_{2k}^{*}(x),G_{2k-1}^{*}(x))=1$. This implies that
$\gcd(G_{2k+1}^{*}(x),G_{2k-1}^{*}(x))=\gcd(d(x),G_{2k-1}^{*}(x))$.
This and Lemma \ref{gcdlemmas} imply that
\[\gcd(G_{2k+1}^{*}(x),G_{2k-1}^{*}(x))=\gcd(d(x),G_{2k-1}^{*}(x))=G_1^{*}(x).\]
This completes the proof of part (1).
\end{proof}

\section{Identities and other properties of generalized Fibonacci polynomials} \label{mainresults:section}

In this section we present some identities that the GFPs satisfy. These identities are required for
the proofs of certain divisibility  properties of the GFPs. The results in this section are proved using
the Binet formulas (see Section \ref{General:Fibonacci:Polynomial}).  Proposition \ref{divisity:Hogat:property1} part (1) is a variation of a
result proved in \cite{HoggattLong}, similarly Proposition \ref{divisity:property:fibonacci} is a variation of a divisibility property
proved by them in the same paper. A collection of this type of identities for GFPs can be found in \cite{florezMcanallyMuk}.

In 1963 Ruggles \cite{koshy} proved that $F_{n+m}= F_{n} L_{m}-(-1)^m F_{n-m}$. Proposition \ref{divisity:Hogat:property1}
parts (2) and (3) is a generalization of this numerical identity. In 1972 Hansen \cite{hansen} proved that
$ 5 F_{m +n - 1}=L_{m} L_{n} + L_{m - 1} L_{n -1} $. Proposition \ref{divisity:Hogat:property2} part (1) is a
generalization of this numerical identity.

\begin{proposition} \label{divisity:Hogat:property1}
If $\{ G_{n}^{*}(x)\}$ and $\{ G_{n}^{\prime}(x)\}$ are equivalent GFPs sequences, then

\begin{enumerate}[(1)]
  \item $G_{m+n+1}^{\prime}(x)= G_{m+1}^{\prime}(x)G_{n+1}^{\prime}(x)+g(x)G_{m}^{\prime}(x)G_{n}^{\prime}(x)$,
  \item if $n\ge m$, then $G_{n+m}^{\prime}(x)= \alpha G_{n}^{\prime}(x) G_{m}^{*}(x)-(-g(x))^{m} G_{n-m}^{\prime}(x)$,
  \item if $n\ge m$, then $G_{n+m}^{\prime}(x)= \alpha G_{m}^{\prime}(x)G_{n}^{*}(x) +(-g(x))^{m} G_{n-m}^{\prime}(x)$.
\end{enumerate}
\end{proposition}

\begin{proof} We prove part (1). We know that $G_{n}^{\prime}(x)$
satisfies the Binet formula (\ref{bineformulauno}). That is, $R_n(x)= (a^{n}-b^{n})/(a-b)$ .
(Recall that we use $a:=a(x)$ and $b:=b(x)$.)

Therefore, $G_{m+1}^{\prime}(x)G_{n+1}^{\prime}(x)+g(x)G_{m}^{\prime}(x)G_{n}^{\prime}(x)$ is equal to,
\[\left[(a^{m+1}-b^{m+1})(a^{n+1}-b^{n+1})+g(x)(a^{m}-b^{m})(a^{n}-b^{n})\right]/(a-b)^{2}.\]
Simplifying and factoring terms we obtain:
\[\left[\left(a^{n+m}(a^2+g(x))+b^{n+m}(b^2+g(x))\right)-(a^nb^m+b^na^m)(ab+g(x))\right]/(a-b)^{2}.\]
Next, since $ab=-g(x)$, we see that the above expression is equal to
\[\left[a^{n+m}(a^2+g(x))+b^{n+m}(b^2+g(x))\right]/(a-b)^{2}.\]
This, with the facts that $(a^2+g(x))=a(a-b)$ and $(b^2+g(x))=-b(a-b)$, shows that the above expression is equal to
\[\left(a^{n+m+1}-b^{n+m+1}\right) /(a-b)=R_{n+m+1}(x).\]
This completes the proof of part (1).

We now prove part (2), the proof of part (3) is identical and we omit it. Suppose that $G_{k}^{*}(x)$ is equivalent to
$G_{k}^{\prime}(x)$ and that $G_{k}^{*}(x)$ is of the Lucas type for all $k$. For simplicity let us suppose that $\alpha=1$
(the proof when $\alpha\neq 1$ is similar, so we omit it).
Using the Binet formulas (\ref{bineformulados}) and (\ref{bineformulauno}) we have that
$G_{n}^{\prime}(x) G_{m}^{*}(x)-(-g(x))^{m} G_{n-m}^{\prime}(x)$ is equal to
\[\displaystyle{\frac{(a^{n}-b^{n})(a^{m}+b^{m})-(-g(x))^{m}(a^{n-m}-b^{n-m})}{(a-b)}}.\]
After performing the indicated multiplication and simplifying we find that this expression is equal to
\[\left[\frac{a^{n+m}-b^{n+m}}{a-b} \right]+ \left[\frac{a^{n}b^{m}-a^{m}b^{n}-(-g(x))^m a^{n-m}+(-g(x))^mb^{n-m}}{a-b}\right].\]
Since $-g(x)=ab$, it is easy to see that the expression in the right bracket is equal to zero. Thus, $(a^{n+m}-b^{n+m})/(a-b)= G_{n+m}^{\prime}(x)$.
This completes the proof of part (2).
\end{proof}

\begin{proposition} \label{divisity:Hogat:property2}
Let $\{ G_{n}^{*}(x)\}$ and $\{ G_{n}^{\prime}(x)\}$ be equivalent GFPs sequences. If $m\ge 0$ and $n\ge 0$, then
\begin{enumerate}[(1)]
\item $(a-b)^2G^{\prime}_{m+n+1}(x) = \alpha^2 G_{m+1}^{*}(x)G_{n+1}^{*}(x)+\alpha^2 g(x)G_{m}^{*}(x)G_{n}^{*}(x),$
\item $G_{m+n+2}^{*}(x) = \alpha G_{m+1}^{*}(x)G_{n+1}^{*}(x)+g(x)[\alpha G_{m}^{*}(x)G_{n}^{*}(x)-G_{m+n}^{*}(x)].$
\end{enumerate}
\end{proposition}

\begin{proof} In this proof we use $\alpha=1$, the proof when $\alpha\neq 1$ is similar, so we omit it.
(Recall, once again, that we use $a:=a(x)$ and $b:=b(x)$.)

Proof of part (1).  Since  $ G_{n}^{*}(x)$  is a GFP of the Lucas type, we have that
$G_{n}^{*}(x)$ satisfies the Binet formula $L_n(x)= (a^{n}+b^{n})/\alpha$ given in (\ref{bineformulados}).
Therefore,
\begin{equation}\label{divisity:Hogat:property2:formula1}
G_{m+1}^{*}(x)G_{n+1}^{*}(x)+g(x)G_{m}^{*}(x)G_{n}^{*}(x)
\end{equation}
 is equal to,
\[\left[a^{n+1}+b^{n+1}\right]\left[a^{m+1}+b^{m+1}\right]+g(x)\left[a^{n}+b^{n}\right]\left[a^{m}+b^{m}\right].\]
Simplifying and factoring we see that this expression is equal to
\[a^{m+n}\left[a^2+g(x)\right]+b^{m+n}\left[b^2+g(x)\right]+(ab+g(x))\left[a^mb^n+a^nb^m\right].\]
Since
\[ ab=-g(x),\; a^2+g(x)=a(a-b),\text{ and } b^2+g(x)=-b(a-b),\]
the expression in (\ref{divisity:Hogat:property2:formula1}) is equal to $(a-b)(a^{m+n+1}-b^{m+n+1})$.
We recall that $G_{m+n+1}^{\prime}(x)$ is equivalent to $G_{m+n+1}^{*}(x)$. Thus, $G^{\prime}_{m+n+1}(x) =(a^{m+n+1}-b^{m+n+1})/(a-b)$.
Therefore,  $(a-b)^2G^{\prime}_{m+n+1}(x) =(a-b)\left[a^{m+n+1}-b^{m+n+1}\right]$. This completes the proof of part (1).

Proof of part (2). From the proof of part (1)  we know that $$(a-b)^2G^{\prime}_{m+n+1}(x)=(a-b)[a^{m+n+1}-b^{m+n+1}].$$
Simplifying the right side of the previous equality we have
\[ (a-b)^2G^{\prime}_{m+n+1}(x)=a^{m+n+2}-ba^{m+n+1}- ab^{m+n+1} + b^{m+n+2}.\]
So, $(a-b)^2G^{\prime}_{m+n+1}(x)=a^{m+n+2}+ b^{m+n+2} -ab[a^{m+n}+ b^{m+n}]$. We recall that $ab=-g(x)$. Thus,
\[(a-b)^2G^{\prime}_{m+n+1}(x)=a^{m+n+2}+ b^{m+n+2} + g(x)[a^{m+n}+ b^{m+n}].\]
This and the Binet formula (\ref{bineformulados}) imply that
\[(a-b)^2G^{\prime}_{m+n+1}(x)=G_{m+n+2}^{*}(x) + g(x)G_{m+n}^{*}(x).\]
So, the proof follows from part (1) of this proposition.
\end{proof}

\begin{proposition} \label{Dic2}
Let $\{ G_{n}^{*}(x)\}$ be a GFP sequence of the Lucas type.
If $m$, $n$, $r$, and $q$ are positive integers, then
\begin{enumerate}[(1)]
\item if $m\leq n$, then $G_{m+n}^{*}(x)=\alpha G_{m}^{*}(x)G_{n}^{*}(x)+(-1)^{m+1}(g(x))^mG_{n-m}^{*}(x)$.
\item If $r<m$, then there is a polynomial $T(x)$ such that
\[ G_{mq+r}^{*}(x)=
    \left\{
         \begin{array}{ll}
            G_{m}^{*}(x)T(x)+(-1)^{w+t}(g(x))^{w} G_{m-r}^{*}(x),       & \hbox{ if } $q$ \hbox{ is odd;} \\
            G_{m}^{*}(x)T(x)+(-1)^{(m+1)t}(g(x))^{mt} G_{r}^{*}(x),     & \hbox{ if }  $q$ \hbox{ is even}
        \end{array}
    \right.
\]
where $t=\left\lceil \frac{q}{2} \right\rceil$ and $w=(t-1)m+r$.
\item If $n>1$, then there is a polynomial $T_n(x)$ such that
\[G_{2^nr}^{*}(x)= G_{r}^{*}(x)T_n(x)+(2/\alpha)(g(x))^{2^{n-1}r}.
\]

\end{enumerate}
\end{proposition}

\begin{proof} We prove part (1). Since $G_{m}^{*}(x)$ and $G_{n}^{*}(x)$ are of the Lucas type, they both satisfy the Binet formula
(\ref{bineformulados}).
Thus,
\[
G_{m}^{*}(x)G_{n}^{*}(x)=\left(\frac{a^m+b^m}{\alpha}\right)\left(\frac{a^n+b^n}{\alpha}\right)
=\frac{a^{m+n}+b^{m+n}}{\alpha^2}+\frac{(ab)^{m}\left(a^{n-m}+b^{n-m}\right)}{\alpha^2}.
\]
So, $G_{m}^{*}(x)G_{n}^{*}(x)=\left[G_{n+m}^{*}(x)+(ab)^mG_{n-m}^{*}(x)\right]/\alpha.$ This and $ab=-g(x)$ imply that
\[G_{m+n}^{*}(x)=\alpha G_{m}^{*}(x)G_{n}^{*}(x)-(-g(x))^{m} G_{n-m}^{*}(x).\]
This completes the proof of part (1).

We prove part 2 using cases and mathematical induction.

{\bf Case $q$ is odd}. Suppose $q=2t-1$, and  let $S(t)$ be the following statement.
For every positive integer $t$ there is a polynomial $T_t(x)$ such that
\[G_{m(2t-1)+r}^{*}(x)=G_{m}^{*}(x)T_t(x)+(-1)^{m(t-1)+t+r}(g(x))^{(t-1)m+r} G_{m-r}^{*}(x).\]
From part (1), taking $T_1(x)=\alpha G_{r}^{*}(x)$, it is easy to see that $S(t)$ is true if $t=1$.

We suppose that $S(k)$ is true. That is, suppose that there is a polynomial $T_{k}(x)$ such that
\begin{equation}\label{formula1:for:Dic2}
G_{m(2k-1)+r}^{*}(x)=G_{m}^{*}(x)T_k(x)+(-1)^{m(k-1)+t+r}(g(x))^{(k-1)m+r} G_{m-r}^{*}(x).
\end{equation}

We prove that $S(k+1)$ is true.
Notice that $G_{m(2k+1)+r}^{*}(x)=G_{(2km+r)+m}^{*}(x)$. Therefore, from part (1) we have
$$
G_{m(2k+1)+r}^{*}(x)=\alpha G_{m}^{*}(x)G_{2km+r}^{*}(x)+(-1)^{m+1}g^m(x)G_{m(2k-1)+r}^{*}(x).
$$
This and $S(k)$ (see (\ref{formula1:for:Dic2})) imply that
\[G_{m(2k+1)+r}^{*}(x)= \alpha G_{m}^{*}(x)G_{2km+r}^{*}(x)+(-1)^{m+1}(g(x))^{m}G_{m}^{*}(x)T_k (x)+M_{1}(x) \]
where $M_{1}(x)=(-1)^{km+(t+1)+r}(g(x))^{km+r} G_{m-r}^{*}(x)$.
Therefore, $G_{m(2k+1)+r}^{*}(x)$ is equal to
\[ G_{m}^{*}(x)[\alpha G_{2km+r}^{*}(x)+(-1)^{m+1}(g(x))^{m}T_k (x)]+(-1)^{km+(t+1)+r}(g(x))^{km+r} G_{m-r}^{*}(x).\]
This, with $T_{k+1}(x):=\alpha G_{2km+r}^{*}(x)+(-1)^{m+1}(g(x))^mT_k (x)$, implies  $S(k+1)$. This completes the proof when $q$
is odd.

{\bf Case $q$ is even}. This proof is similar to the case in which $q$ is odd.
Suppose $q=2t$, and  let $H(t)$ be the following statement. For every positive integer there is a polynomial $T_t(x)$
such that
\[G_{m(2t)+r}^{*}(x)=G_{m}^{*}(x)T_t(x)+(-1)^{(m+1)t}(g(x))^{mt} G_{r}^{*}(x).\]
From part (1), taking $T_1(x)=\alpha G_{r}^{*}(x)$, it is easy to see that $H(t)$ is true if $t=1$.

We suppose that $H(k)$ is true. That is, suppose that there is a polynomial $T_{k}(x)$ such that
\begin{equation}\label{formula2:for:Dic2}
G_{m(2k)+r}^{*}(x)=G_{m}^{*}(x)T_k(x)+(-1)^{(m+1)k}(g(x))^{mk} G_{r}^{*}(x).
\end{equation}

We prove that $H(k+1)$ is true.
Notice that $G_{2m(k+1)+r}^{*}(x)=G_{((2k+1)m+r)+m}^{*}(x)$. Therefore, from part (1) we have
$$
G_{2m(k+1)+r}^{*}(x)=\alpha G_{m}^{*}(x)G_{(2k+1)m+r}^{*}(x)+(-1)^{m+1}(g(x))^m G_{2mk+r}^{*}(x).
$$
This and $H(k)$ (see (\ref{formula2:for:Dic2}))  imply that
\[ G_{m(2(k+1))+r}^{*}(x)=\alpha G_{m}^{*}(x)G_{(2k+1)m+r}^{*}(x)+(-1)^{m+1}(g(x))^mG_{m}^{*}(x)T_{k} (x)+M_2(x)\]
where $M_{2}(x)=(-1)^{(k+1)(m+1)}(g(x))^{m(k+1)}G_{r}^{*}(x)$.
Therefore,
\[G_{m(2(k+1))+r}^{*}(x)=G_{m}^{*}(x) \left[\alpha G_{(2k+1)m+r}^{*}(x)+(-1)^{m+1}(g(x))^mT_{k} (x)\right]+M_{2}(x).\]
This, with $T_{k+1}(x):=\alpha G_{(2k+2)m+r}^{*}(x)+(-1)^{m}(g(x))^mT_{k} (x)$, implies  $H(k+1)$.

We finally prove part (3) by induction.  Since $G_{n}^{*}(x)$ is of the Lucas type, by the Binet formula it is easy to see that
$G_0(x)=2/\alpha$.
Let $S(n)$ be the statement: for every positive integer $n$ there is a polynomial $T_n(x)$ such that this equality holds
$G_{2^nr}^{*}(x)= G_{r}^{*}(x)T_n(x)+(2/\alpha)g^{2^{n-1}r}(x)$.

Proof of $S(2)$. From part (1) we have
$$G_{2^2r}^{*}(x)= G_{2r+2r}^{*}(x)=\alpha(G_{2r}^{*}(x))^2-(2/\alpha)(g(x))^{2r}.$$
Applying the result in part (1) for $G_{2r}^{*}(x)$ again (and simplifying) we obtain:

\begin{tabular}{rl}
	$G_{2^2r}^{*}(x)$=&$\alpha[\alpha(G_{r}^{*}(x))^2-\dfrac{2}{\alpha}(-g(x))^{r}]^2-\dfrac{2}{\alpha}(g(x))^{2r}$\\ [10pt]
	=&$G_{r}^{*}(x)[(\alpha G_{r}^{*}(x))^3-(-1)^{r}4\alpha G_r^{*}(x)(g(x))^{r}]+\frac{4(g(x))^{2r} -2(g(x))^{2r}}{\alpha}$\\ [10pt]
	=&$G_{r}^{*}(x)T_2(x)+\dfrac{2}{\alpha}(g(x))^{2r}$	
\end{tabular}

\noindent 	 where $T_2(x)=\alpha^3(G_{r}^{*}(x))^3+(-1)^{r+1}4\alpha G_{r}^{*}(x)(g(x))^{r}$. This proves $S(2)$.

We suppose that $S(k)$ is true for $k>2$, and we prove $S(k+1)$ is true. That is, we suppose that for a fixed $k$ there is a
polynomial $T_k(x)$ such that
\[G_{2^kr}^{*}(x)= G_{r}^{*}(x)T_k(x)+(2/\alpha)g^{2^{k-1}r}(x).\] From part (1) we have
$G_{2^{k+1}r}^{*}(x)= G_{2^k r+2^k r}^{*}(x)=\alpha(G_{2^k r}^{*}(x))^2-(2/\alpha)(g(x))^{2^k r}$.
Using the result from the inductive hypothesis $S(k)$ and simplifying, we obtain:

\begin{tabular}{rl}
	$G_{2^{k+1}r}^{*}(x)$=&$\alpha[ G_{r}^{*}(x)T_k(x)+\dfrac{2}{\alpha}(g(x))^{2^{k-1} r} ]^2-\dfrac{2}{\alpha}(g(x))^{2^k r}$\\ [10pt]
	=&$G_{r}^{*}(x)[\alpha G_{r}^{*}(x)T_k ^2(x)+4T_k(x)(g(x))^{2^{k-1} r} ]+\frac{4(g(x))^{2^k r} -2(g(x))^{2^kr}}{\alpha}$\\ [10pt]
	=&$G_{r}^{*}(x)T_{k+1}(x)+\dfrac{2}{\alpha}(g(x))^{2^kr}$\\
\end{tabular}

\noindent  where $T_{k+1}(x)=\alpha G_{r}^{*}(x)T_k ^2(x)+4T_k(x)(g(x))^{2^{k-1} r} $. This completes the proof of part (3).
\end{proof}

In the following part of this section, we present two divisibility properties for the GFPs. Proposition \ref{divisity:property:fibonacci} is \cite[Theorem 6]{HoggattLong}.

\begin{proposition} \label{divisity:property:fibonacci}
If $\{ G_{n}^{\prime}(x)\}$ is a GFP sequence of the Fibonacci type, then
$G_{m}^{\prime}(x)$ divides $G_{n}^{\prime}(x)$ if and only if $m$ divides $n$.
\end{proposition}

\begin{proof}
We first prove the sufficiency. Based on the hypothesis that $m$ divides $n$,
there is an integer $q\ge1$ such that $n=mq$. Then, using the Binet formula (\ref{bineformulauno}), we have
\[G_{m}^{\prime}(x)=(a^m-b^m)/(a-b)\,\, \text{  and  }\,\, G_{mq}^{\prime}(x)=(a^{mq}-b^{mq})/(a-b). \]

It is easy to see --using induction on $q$-- that $(a^m-b^m)$ divides $(a^{mq}-b^{mq})$ which implies that $G_{m}^{\prime}(x)$ divides
$G_{mq}^{\prime}(x)$. This proves the sufficiency.

We now prove the necessity. Suppose that $m$ does not divide $n$ and that $G_{m}^{\prime}(x)$ divides $G_{n}^{\prime}(x)$ for $m$ and $n$ greater than $1$.
Therefore, there are integers $q$ and $r$ with $0<r<n$ such that $n=mq+r$. Then by Proposition \ref{divisity:Hogat:property1} part (1)
\begin{eqnarray*}
  G_{n}^{\prime}(x) &=& G_{mq+r}^{\prime}(x) \\
   &=& G_{mq+1}^{\prime}(x)G_{r}^{\prime}(x)+g(x)G_{mq}^{\prime}(x)G_{r-1}^{\prime}(x)\\
   &=& \left(d(x)G_{mq}^{\prime}(x)+g(x)G_{mq-1}^{\prime}(x)\right)G_{r}^{\prime}(x)+g(x)G_{mq}^{\prime}(x)G_{r-1}^{\prime}(x)\\
   &=& d(x) G_{mq}^{\prime}(x) G_{r}^{\prime}(x)+g(x)G_{mq-1}^{\prime}(x)G_{r}^{\prime}(x)+g(x)G_{mq-1}^{\prime}(x)G_{r}^{\prime}(x).
\end{eqnarray*}
Grouping terms and simplifying we obtain,
\[G_{n}^{\prime}(x)=G_{mq}^{\prime}(x)G_{r+1}^{\prime}(x)+g(x)G_{mq-1}^{\prime}(x)G_{r}^{\prime}(x).\]
This and the fact that $G_{m}^{\prime}(x)\mid G_{n}^{\prime}(x)$ and $G_{m}^{\prime}(x) \mid G_{mq}^{\prime}(x)$ imply that
$$G_{m}^{\prime}(x)\mid g(x)G_{mq-1}^{\prime}(x)G_{r}^{\prime}(x).$$
From Lemma \ref{gcdlemmas} part (3) and Proposition \ref{gcddistance1;2} we know that $\gcd(G_{mq}^{\prime}(x),g(x))=1$
and that $\gcd(G_{mq-1}^{\prime}(x), G_{mq}^{\prime}(x))=1$, respectively.
These two facts imply that $G_{m}^{\prime}(x)\mid G_{r}^{\prime}(x).$
 That is a contradiction since degree ($G_{r-1}^{\prime}(x))<$ degree $(G_{m-1}^{\prime}(x))$. This completes the proof.
\end{proof}

The following corollary gives a factorization of a GPF of Fibonacci type $G_{n}^{\prime}(x)$. It is a direct application of
Theorem \ref{divisity:property:fibonacci}  and some results given in articles by Bliss, et al. and Nowicki \cite{bliss, nowicki}.
In fact, the proof of Corollary \ref{corollary:nowicki} follows from Theorem \ref{divisity:property:fibonacci}
and \cite[Theorem 2]{nowicki}, so we omit the details. We start by giving a short background of the lcm sequences for
polynomials that fits the context of this paper, a more general result can be found in  \cite{nowicki}.
If $\lcm$ denotes the least common multiple, then we define $c_{n}(x)$
recursively as follows:  Let $c_{1}=1$ and
$$c_{n}(x)=\frac{\lcm(G_{1}^{\prime}(x), G_{2}^{\prime}(x), \dots, G_{n}^{\prime}(x))}{\lcm(G_{1}^{\prime}(x), G_{2}^{\prime}(x), \dots, G_{n-1}^{\prime}(x))} \quad \text{ for } n\ge 2.$$

\begin{corollary} \label{corollary:nowicki} If $G_{n}^{\prime}(x)$ is a GFP of Fibonacci type, then $G_{n}^{\prime}(x)=\prod_{d|n}c_d(x)$.
\end{corollary}

The factors $c_{i}$ are not always irreducible polynomials. For instance, if $G_{10}^{\prime}(x)$ is a
Chebyshev polynomial of the first kind, then its irreducible factoring is
 $$G_{10}^{\prime}(x) = 2x(-1-2x+4x^2) (-1+2x+4x^2) (5-20x^2 +16x^4).$$
 Using Corollary \ref{corollary:nowicki} we have
 \begin{eqnarray*}
G_{10}^{\prime}(x) &=& c_1(x)c_{2}(x)c_{5}(x)c_{10}(x)\\
&=&(1)(2x)(1-12x^2+16x^4) (5-20x^2 +16x^4).
\end{eqnarray*}

\begin{proposition} \label{divisibity:property:first:type} Let $m$ be a positive integer that is not a power of two.
If $G_{m}^{*}(x)$ is a GFP of Lucas type, then for all  odd divisors $q$ of $m$, it holds that
$G_{m/q}^{*}(x)$ divides $G_{m}^{*}(x)$. Moreover $G_{m/q}^{*}(x)$ is of the Lucas type.
\end{proposition}

\begin{proof} Let $q$ be an odd divisor of $m$. If $q=1$ the result is obvious. Let us suppose that $q\not = 1$. Therefore, there
is an integer $d>1$ such that $m=dq$. Using the Binet formula (\ref{bineformulados}), where $a:=a(x)$ and $b:=b(x)$, we have
$G_{m}^{*}(x)=G_{dq}^{*}(x)=(a^{dq}+b^{dq})/\alpha$. Let $X=a^{d}$ and $Y=b^{d}$. Using induction it is possible to prove that $X+Y$
divides $X^q+Y^q$. This implies that there is a
polynomial $Q(x)$ such that $(X^q+Y^q)/\alpha=Q(x)(X+Y)/\alpha$. Therefore,
\[G_{m}^{*}(x)=G_{dq}^{*}(x)=(a^{dq}+b^{dq})/\alpha=Q(x)(a^{d}+b^{d})/\alpha.\]
This and the Binet formula (\ref{bineformulados}) imply that $G_{m}^{*}(x)=G_{d}^{*}(x) Q(x)$.
\end{proof}

\section{Characterization of the strong divisibility property}\label{gcd:characterization}

In this section we prove the main results of this paper. Thus, we prove a necessary and sufficient condition for the
strong divisibility property for GFPs of Fibonacci type. We also prove that the
strong divisibility property holds partially for GFPs of Lucas type. The other important result in this
section is that the strong divisibility property holds partially for a GFP and its equivalent.
The results here therefore provide a complete characterization of the strong divisibility property satisfied by the GFPs
of Fibonacci type.

We note that if $G_{m}^{*}(x)$ and  $G_{n}^{\prime}(x)$ are two equivalent polynomials from Table \ref{equivalent}, then $$\gcd(G_{m}^{*}(x), G_{n}^{\prime}(x)) \text{ is either } G_{\gcd(m,n)}^{*}(x) \text{ or one.}$$
However, it is not true in general. Here we give an example of a pair of GFPs that do not satisfy this property.
First we define a Fibonacci type polynomial
\[G_0^{\prime}(x)=0, \; G_1^{\prime}(x)= 1,\;  \text{and} \;  G_{n}^{\prime}(x)= (2x+1) G_{n - 1}^{\prime}(x) + G_{n - 2}^{\prime}(x) \text{ for } n\ge 2.\]
We now define its equivalent polynomial of the Lucas type
\[G_0^{*}(x)=2, \; G_1^{*}(x)= 2x+1,\;  \text{and} \;  G_{n}^{*}(x)= (2x+1) G_{n - 1}^{*}(x) + G_{n - 2}^{*}(x) \text{ for } n\ge 2.\]
After some calculations we see that $\gcd(G_{m}^{*}(x), G_{n}^{\prime}(x))$ is one, two, or $G_{\gcd(m,n)}^{*}(x)$.
Using the same polynomials we can also see that $\gcd(G_{m}^{*}(x), G_{n}^{*}(x))$ is one, two, or $G_{\gcd(m,n)}^{*}(x)$.
If we do the same calculations for numerical sequences (Fibonacci and Lucas numbers), we can see that they have the same behavior.

In this section we use the notation $E_{2}(n)$  to represent the \emph{integer exponent base two} of a
positive integer $n$  which is defined to be the largest integer $k$ such that $2^{k}\mid n$.

\begin{lemma}\label{Dic1}
If $R(x)$, $S(x)$, and $T(x)$ are polynomial in $\mathbb{Z}[x]$, then
\[\gcd(R(x),T(x))=\gcd(R(x),R(x)S(x)-T(x)).\]
\end{lemma}

\begin{proposition}\label{propiedadDivision}
Let $\{ G_{n}^{*}(x)\}$ be a GFP sequence of the Lucas type.
If $m \mid n$ and $E_2(n)=E_{2}(m)$, then
$
\gcd(G_{n}^{*}(x),G_{m}^{*}(x))=G_{m}^{*}(x).
$
\end{proposition}

\begin{proof} First we recall that $E_{2}(n)$ is the largest integer
$k$ such that $2^{k}\mid n$.  We suppose that $n=mq$ with $q\in\mathbb{N}$.
Since $E_{2}(m)=E_{2}(n)=E_{2}(mq)$,
we conclude that $q$ is odd. This, Lemma \ref{Dic1}, and
Proposition \ref{Dic2} part (2) imply that
\begin{eqnarray*}
\gcd(G_{n}^{*}(x),G_{m}^{*}(x))&=&\gcd(G_{qm}^{*}(x),G_{m}^{*}(x))\\
&=&\gcd (G_{m}^{*}(x)T(x)+(-1)^{n}(-g(x))^{(n-1)m} G_{m}^{*}(x),G_{m}^{*}(x))\\
&=&G_{m}^{*}(x).
\end{eqnarray*}
This proves the proposition.
\end{proof}

\begin{corollary} \label{divisity:property:firstype} Let $G_{m}^{*}(x)$ be a GFP of Lucas type.
If $m>0$ is not a power of two, then for all odd divisors $q$ of $m$, it  follows that
$G_{m/q}^{*}(x)$ divides $G_{m}^{*}(x)$. More over $G_{m/q}^{*}(x)$ is of the Lucas type.
\end{corollary}

\begin{proof}
It is easy to see that $E_2(m/q)=E_2(m)$. Therefore, the conclusion follows by Proposition \ref{propiedadDivision}.
\end{proof}

\begin{proposition}\label{propiedadDivisioncase2}
Let $d_{k}=\gcd(G_0 ^*(x),G_{k} ^{*}(x))$ where  $G_{k}^{*}(x)$ is a GFP of the Lucas type. Suppose that there is an integer  $k'>0$  such that $d_{k'}=2$.
If  $m$ is the minimum positive integer such that $d_{m}=2$, then $m|n$ if and only if $d_{n}=2$.
\end{proposition}

\begin{proof} We suppose that $m$ is the minimum positive integer such that $d_{m}=2$.
Let $m|n$, by  Proposition \ref{propiedadDivision} we know that
$\gcd(G^* _{m}(x),G_{n}^{*}(x))=G_{m}^{*}(x)$ (we recall that $G^*_{0}(x)=p_0(x)$ and $|p_0(x)|=1$ or $2$).
This and the fact that $2|G_{m} ^*(x)$ imply that  $\gcd(G_{0} ^*(x),G_{n} ^*(x))=2$. This proves that   $d_{n}=2$.

We note that $2|\gcd(G^{*}_{m}(x),G_{n} ^*(x))$. Suppose that there is a $n\in \mathbb{N}-\{m\}$ that satisfies the condition
$d_{n}=2$.
From the division algorithm we have that there are integers $q$ and $r$ such that $n=mq+r$ where $0\le r< m$.
This and Proposition \ref{Dic2} part (2) imply that
\[ \gcd(G^{*} _{m}(x),G_{n} ^*(x))=
    \left\{
         \begin{array}{ll}
            \gcd(G^* _{m}(x),(g(x))^{(t-1)m+r} G_{m-r}^{*}(x)), 	& \hbox{ if } $q$ \hbox{ is odd;} \\
            \gcd(G^* _{m}(x),(g(x))^{mt} G_{r}^{*}(x)),             & \hbox{ if }  $q$ \hbox{ is even.}
        \end{array}
    \right.
\]
This and Lemma \ref{gcdlemmas} part (3) imply that $\gcd(G^* _{m}(x),G_{n} ^*(x))$ is either
$$\gcd(G^* _{m}(x),G_{r}^{*}(x)) \quad \text{ or } \quad \gcd(G^* _{m}(x),G_{m-r}^{*}(x)).$$
From this and the fact that  $2|\gcd(G^* _{m}(x),G_{n} ^*(x))$,
we conclude that either
$$\gcd(G^* _{r}(x),G^* _{0}(x))=2 \quad \text{ or } \quad \gcd(G^* _{m-r}(x),G^*_{0}(x))=2.$$ This holds only if $r=0$, due to definition of $m$.
Therefore, $n=mq$.
\end{proof}

\begin{lemma}\label{fundamental}
Let $G_{k}^{*}(x)$ be a GFP of Lucas type and let $n=mq+r$ where $m, q$ and $r$ are positive integers with $r<m$.
If  $m_1= m-r$ when $q$  is odd  and  $m_1=r$ when $q$ is even, then
  $
\gcd(G_{n}^{*}(x),G_{m}^{*}(x))=\gcd (G_{m_1}^{*}(x),G_{m}^{*}(x)).
$
\end{lemma}

\begin{proof} Let
 \[
f(x)=
 \begin{cases}
         (-1)^{m(t-1)+t+r}(g(x))^{(t-1)m+r} , & \mbox{ if  $q$ is odd;}\\
        (-1)^{(m+1)t}(g(x))^{mt},  & \mbox{ if  $q$   is even.}
\end{cases}
 \]
This and Lemma \ref{gcdlemmas} part (3) imply that  $\gcd(G_{m}^{*}(x),f(x))=1$.
Therefore, by Proposition \ref{Dic2} part (2) it follows that
$$\gcd(G_{mq+r}^{*}(x),G_{m}^{*}(x))=\gcd (G_{m}^{*}(x)T(x)+f(x)G_{m_1}^{*}(x),G_{m}^{*}).$$ Now it is easy to see that
\[\gcd (G_{m}^{*}(x)T(x)+f(x)G_{m_1}^{*}(x),G_{m}^{*})=\gcd (f(x) G_{m_1}^{*}(x),G_{m}^{*}(x)).\]
Since $\gcd(G_{m}^{*}(x),f(x))=1$, by Proposition  \ref{prop2;1} part (1) we have
$$\gcd (f(x) G_{m_1}^{*}(x),G_{m}^{*}(x))=\gcd (G_{m_1}^{*}(x),G_{m}^{*}(x)).$$
This completes the proof.
\end{proof}

\begin{theorem} \label{second:main:thm}
Let $ G_{n}^{*}(x)$ be a GFP of the Lucas type.
If $m$ and $n$ are positive integers and $d=\gcd(m,n)$, then
\[ \gcd(G_{m}^{*}(x),G_{n}^{*}(x))=
		\begin{cases} G_{d}^{*}(x),					 & \mbox{ if }\;   E_{2}(m)= E_{2}(n);\\
                      \gcd(G_{d}^*(x),G_{0}^*(x)),   & \mbox{ otherwise}.
           \end{cases}
\]
\end{theorem}

\begin{proof}
First we prove the result for $E_{2}(n)=E_{2}(m)$. From the Euclidean algorithm we know that there are non-negative
integers $q$ and $r$ such that $n=mq+r$ with $r<m$. Let  $d=\gcd(m,n)$. Clearly, if $r=0$, then $d=m$. Therefore, the
result holds by Proposition  \ref{propiedadDivision}.

We suppose that
$r\neq 0$. If we take $m_1$ as in Lemma \ref{fundamental}, then
$$
\gcd (G_{n}^{*}(x),G_{m}^{*}(x))=\gcd(G_{mq+r}^{*}(x),G_{m}^{*}(x))=\gcd (G_{m_1}^{*}(x),G_{m}^{*}(x)).
$$
Let $M_1=\{m,m_1 \}$. Notice that $\gcd(m_1,m)=d$, $E_2(m)=E_2(m_1)$, and that $m_1<m$. Therefore, there are non-negative integers
$q_1$  and $r_1$ such that  $m=m_1q_1+r_1$ with $r_1<m_1$. Again, if  $r_1=0$, by Proposition \ref{propiedadDivision}
we obtain that 	
$
\gcd(G_{n}^{*}(x),G_{m}^{*}(x))=\gcd(G_{m}^{*}(x),G_{m_1}^{*}(x))=G_{d}^{*} (x).
$
If $r_1\neq 0$ we repeat the previous step and then we can guarantee that
$$
\gcd (G_{n}^{*}(x),G_{m}^{*}(x))=\gcd (G_{m_1}^{*}(x),G_{m}^{*}(x))=\gcd (G_{m_1}^{*}(x),G_{m_2}^{*}(x)),
$$
where
\[ m_2=\begin{cases}m_1-r_1, 	& \mbox{ if $q$ is odd};\\
   	  				r_1,  		& \mbox{ if } \mbox{ $q$ is even}.
             \end{cases}
\]
We repeat this procedure $t$ times until we obtain the ordered decreasing sequence $m>m_1>m_2>\cdots >m_t\geq d$ such that
$E_2(m)=E_2(m_t)$ and $\gcd(m_t,m_{t-1})=d$, where
\[ m_t=\begin{cases}m_{t-1}-r_{t-1}, 	& \mbox{ if $q$ is odd};\\
   	  				r_{t-1},  			& \mbox{ if } \mbox{ $q$ is even}.
             \end{cases}
\]
Notice that
$M_t=\{m, m_1, m_2, \ldots, m_t\} =M_{t-1}\cup\{m_t\}$ is an ordered set of natural numbers, therefore there is a minimum element.
Since $M_t$ is constructed with  a sequence of decreasing positive integers, there must be an integer $k$ such that
$M_{t}\subset M_k$ for all $t<k$ and $M_{k+1}$ is  undefined. Thus, the procedure ends with $M_k$.  Note that
$m>m_1>m_2>\cdots >m_k \geq d$ such that $E_2(m)=E_2(m_k)$ and $\gcd(m_k,m_{k-1})=d$.

{\bf Claim}. The minimum element of $M_k$ is $m_k=d$ and $m_{k}\mid m_{k-1}$.

Proof of claim. From the division algorithm we know that there are non-negative integers $q_k$ and $r_k$ such that
$m_{k-1}=m_kq_k+r_k$ with $r_k<m_k$. If $r_k\neq 0$ we can repeat the procedure described above to obtain a new set
$M_{k+1}$ with $M_{k} \subset M_{k+1}$.  That is a contradiction. Therefore, $r_k=0$. So, $m_{k-1}=m_kq_k$. This implies that
$\gcd(m_k,m_{k-1})=d$. Thus, $m_k=d$. This proves the claim.

The Claim and Proposition \ref{propiedadDivision}  allow us to conclude that
$$
\gcd (G_{n}^{*}(x),G_{m}^{*}(x))=\gcd (G_{m}^{*}(x),G_{m_1}^{*}(x))=\cdots=\gcd (G_{m_{k-1}}^{*}(x),G_{m_{k}}^{*}(x))=G_{d}^{*}.
$$

We now prove by cases that: if  $E_{2}(n)\neq E_{2}(m)$ and  $d=\gcd(n,m)$, then
$\gcd (G_{n}^{*}(x),G_{m}^{*}(x))=\gcd (G_{d}^{*}(x),G_{0}^{*}(x))$.

\textbf{Case 1.} Suppose that $m<n$ and that $E_2(n)<E_2(m)$. From the division algorithm there are two non-negative
integers $q$ and $r$ such that $n=mq+r$ with $r<m$. Let $m_1= m-r$ when $q$  is odd  and  $m_1=r$ when $q$ is even
(as defined in Lemma \ref{fundamental}). Since $n=mq+r$ and $E_2(n)<E_2(m)$, we have $r\neq 0$. It is
easy to see that $E_2(n)=E_2(r)$, and therefore $E_2(n)=E_2(m_1)$. This and Lemma \ref{fundamental} imply that $
\gcd (G_{n}^{*}(x),G_{m}^{*}(x))=\gcd (G_{m}^{*}(x),G_{m_1}^{*}(x)).
$
Since $E_2(m_1)=E_2(n)<E_2(m)$ and  $m_1<m$, the criteria for the Case 2 are satisfied here, so the proof of this case may be
completed as we are going to do in Case 2 below.

\textbf{Case 2.} Suppose that $E_2(m)<E_2(n)$ and that $m<n$. From the division algorithm we know
that there are two non-negative integers $r$ and $q$ such that  $n=mq+r$ with $ r<m$. If $r=0$, then $q$
must be even (because $E_2(m)<E_2(n)$). Let $k=E_2(q)$ and we consider two subcases on $k$.

	\textbf{Subcase 1.} If $k=1$, then $q=2t$ where $t$ is odd. Therefore,  by Proposition
	\ref{Dic2} part (1) we have
	$
	G_{n}^{*}(x)=G_{2mt}^{*}(x)=\alpha (G_{mt}^{*}(x))^2+(-1)^{mt+1}(G_0 ^*(x))(-g(x))^{mt}.
	$
	This, Proposition \ref{propiedadDivision}, Lemma \ref{gcdlemmas} part (3), and Lemma \ref{Dic1} imply that

    \begin{eqnarray*}
    \gcd (G_{n}^{*}(x),G_{m}^{*}(x))&=& \gcd\left(\alpha (G_{mt}^{*}(x))^2+(-1)^{mt+1}G_0 ^*(x)(-g(x))^{mt}, G_{m}^{*}(x)\right) \\
                            &=& \gcd((-1)^{mt+1}G_0 ^*(x)(-g(x))^{mt},G_{m}^{*}(x)) \\
														&=& \gcd(G_0 ^*(x),G_{m}^{*}(x))\\
														&=& \gcd(G_0 ^*(x),G_{d}^{*}(x)).
    \end{eqnarray*}
	
		\textbf{Subcase 2.}
    If $k>1$, then $q=2^kt$ where $t$ is odd. Therefore, by Proposition \ref{Dic2}
	part (3), there is a polynomial $T_k(x)$ such that
	$$
	G_{n}^{*}=G_{2^kmt}^{*}=G_{mt}^{*}(x)T_k(x)+G_0 ^*(x)g^{2^{k-1}mt}(x).
	$$
	This, Proposition \ref{propiedadDivision}, Lemma \ref{gcdlemmas} part (3), and Lemma \ref{Dic1} imply that
   \begin{eqnarray*}
    \gcd (G_{n}^{*}(x),G_{m}^{*}(x))&=& \gcd(G_{mt}^{*}(x)T_k(x)+G_0 ^*(x) g^{2^{k-1}mt}(x) , G_{m}^{*}(x)) \\
                            &=& \gcd(G_0 ^*(x) g^{2^{k-1}mt}(x) ,G_{m}^{*}(x)) \\
                            &=&\gcd(G_0 ^*(x) ,G_{m}^{*}(x))\\
														&=& \gcd(G_0 ^*(x),G_{d}^{*}(x)).
    \end{eqnarray*}

Now suppose that $r\neq 0$. This and Lemma \ref{fundamental} imply that
$$
\gcd (G_{n}^{*}(x),G_{m}^{*}(x))=\gcd (G_{m}^{*}(x),G_{m_1}^{*}(x)),
$$
where  $m_1= m-r$ when $q$  is odd  and  $m_1=r$ when $q$ is even (as defined in Lemma \ref{fundamental}).
Therefore, $m_1<m<n$ and $\gcd(m,n)=\gcd(m,m_1)=d$.

We analyze both the case in which $m_1\mid m$ and the case in which $m_1 \nmid m$. Suppose that $m=m_1 q_2$
and we consider two cases for $q_2$.

\textbf{Subcase $q_2$ is odd.} If $q_2$ is odd then we have  $E_2(m_1)=E_2(m)$. Therefore, by Proposition \ref{propiedadDivision}
we obtain:
\[\gcd(G_{m}^{*}(x),G_{n}^{*}(x))=\gcd(G_{m}^{*}(x),G_{m_1}^{*}(x))=G_{d}^{*}(x)\; \text{ and }\; E_{2}(G_{d}^{*}(x))<E_{2}(G_{n}^{*}(x)).\]
This implies that $\gcd(G_{m}^{*}(x),G_{n}^{*}(x))=\gcd(G_{d}^{*}(x),G_{0}^{*}(x))$.

\textbf{Subcase $q_2$ is even.} If $q_2$ is even, then $E_2(m_1)<E_{2}(m)$. Now it is easy to see that  $\gcd(G_{m}^{*}(x),G_{n}^{*}(x))=\gcd(G_{m_1}^{*}(x),G_{0}^{*}(x))=\gcd(G_{d}^{*}(x),G_{0}^{*}(x))$.

Now suppose that $m_1 \nmid m$. Therefore there are two non-negative integers
$r_2$ and $q_2$ such that $m=m_1q_2+r_2$ where $0<r_2<m_1$. From Lemma \ref{fundamental} we guarantee
that we can find $m_2$ such that
$$m_2<m_1, \gcd(m_1,m_2)=d \quad \text{ and } \quad \gcd(G_{m_1}^{*}(x),G_{m_2}^{*}(x))=\gcd(G_{m_1}^{*}(x),G_{m}^{*}(x)).$$
In this way we construct an ordered set of integers $M_t= \{n, m, m_1, m_2, \ldots, m_t\}$
where $n> m>m_1>\cdots >m_t$ such that $\gcd(m_j,m_{j-1})=d$ and
$$
\gcd(G_{n}^{*}(x),G_{m}^{*}(x))=\gcd(G_{m_1}^{*}(x),G_{m}^{*}(x))=\cdots=
\gcd(G_{m_j}^{*}(x),G_{m_{j-1}}^{*}(x)).
$$
From Lemma \ref{fundamental} we know that $n> m>m_1>\cdots >m_j$ ends only if $r_j=0$. Since
$M_j= \{n, m, m_1, m_2, \ldots, m_j\}$ is an ordered sequence of natural numbers, it has a minimum element $m_j$.
 Therefore, $m_{j}\mid m_{j-1}$. It is easy then to see that
$$\gcd(G_{n}^{*}(x),G_{m}^{*}(x))=\gcd(G_{m_j}^{*}(x),G_{m_{j-1}}^{*}(x)).$$ This is equivalent to
$\gcd(G_{n}^{*}(x),G_{m}^{*}(x))=\gcd(G_{d}^{*}(x),G_{0}^{*}(x)),
$
which completes the proof.
\end{proof}

\begin{corollary}
Let $d_{k}=\gcd(G_0 ^*(x),G_{k} ^{*}(x))$ where  $G_{k}^{*}(x)$ is a GFP of the Lucas type.
If $m$ and $n$ are positive integers such that $E_{2}(n)\ne E_{2}(n)$, then the following properties hold
\begin{enumerate}[(1)]
  \item Suppose that there is an integer  $k'>0$  such that $d_{k'}=2$.
  If  $r$ is the minimum positive integer such that $d_{r}=2$, then
    \[ \gcd(G_{m}^{*}(x),G_{n}^{*}(x))= \begin{cases} 	2, 	& \mbox{ if }\;   r|\gcd(m,n);\\
                                                    	1,   & \mbox{ otherwise}.
                                        \end{cases}
    \]
  \item If $d_{k}\ne 2$ for every positive integer $k$, then $\gcd(G_{m}^{*}(x),G_{n}^{*}(x))=1$.
\end{enumerate}

\end{corollary}

The proof of this corollary is straightforward from Proposition \ref{propiedadDivisioncase2}.

\begin{proposition}\label{gcd:fibonaccilucas:samegcdlucaslucas}
Let $G_{n}^{*}(x)$ and $G_{n}^{\prime}(x)$ be equivalent GFPs. If $m$ and $n$ are positive integers, then
\begin{enumerate}[(1)]
   \item $\gcd(G^{\prime}_{m+n+1}(x),G_{n}^{*}(x))=\gcd(G_{m+1}^{*}(x),G_{n}^{*}(x)),$
   \item if $m>n$, then $\gcd(G^{\prime}_{m-n+1}(x),G_{n}^{*}(x))=\gcd(G_{m+1}^{*}(x),G_{n}^{*}(x)),$
   \item if $m<n$, then $\gcd(G^{\prime}_{n-m+1}(x),G_{n}^{*}(x))=\gcd(G_{m-1}^{*}(x),G_{n}^{*}(x)).$
\end{enumerate}
\end{proposition}

\begin{proof} We prove part (1) by induction. Let $S(n)$ be the statement (recall that $a-b=a(x)-b(x)$):
for every $n\ge 1$ $\gcd((a-b)^2,G_{n}^{*}(x))=1$. Recall that in a GFP of Lucas type
$\gcd(p_0(x), p_1(x))=\gcd(p_0(x), d(x))=1$ and that $2p_1(x)=p_0(x)d(x)$. From this and
Proposition \ref{divisity:Hogat:property2} part (1) with $m=n=0$, it is easy to see that
$\gcd((a-b)^2,G_{1}^{*}(x))=1$. We now prove that
$S(2)$ is also true. It is easy to see that
\begin{eqnarray*}
\gcd((a-b)^2,G_{2}^{*}(x))&=&\gcd(a^2(x)+b^2(x)-2ab,G_2^{*}(x))\\
&=&\gcd(G_{2}^{*}(x)+2g(x),G_{2}^{*}(x))\\
&=&\gcd(2g(x),G_{2}^{*}(x)).
\end{eqnarray*}

From Lemma (\ref{gcdlemmas}) part (3) we know that $\gcd(g(x),G_{2}^{*}(x))=1$. This implies that either
$$ \gcd((a-b)^2,G_{2}^{*}(x))=1 \; \text{ or } \; \gcd((a-b)^2,G_{2}^{*}(x))=2.$$
If  $\gcd((a-b)^2,G_{2}^{*}(x))=2$, then  $2\mid \left(d^2(x)+4g(x)\right)$ and $2\mid G_{2}^{*}(x)$. So,
$2\mid d^2(x)$ and  $2\mid G_{2}^{*}(x)$. From Lemma (\ref{gcdlemmas}) part (2) we know that $\gcd(d(x),G_{2}^{*}(x))=1$.
This implies that $2 \mid 1$. Therefore, $\gcd((a-b)^2,G_{2}^{*}(x))=1$.  This proves $S(2)$.

Suppose that $S(k)$  is true. Then
$\gcd((a-b)^2,G_{k}^{*}(x))=1$. We now prove that  $S(k+1)$ is true. Suppose that
$\gcd((a-b)^2,G_{k+1}^{*}(x))=r(x)$.  Therefore,  $r(x)\mid(a-b)^2$ and $r(x)\mid G_{k+1}^{*}(x)$.
So, $r(x) \mid [(a-b)^2G^{\prime} _{2k+1}(x)-\alpha^2(G^*_{k+1}(x))^2]$.
From Proposition \ref{divisity:Hogat:property2}  part (1) we know that if $m=n=k$, then
$$
(a-b)^2G^{\prime} _{k+k+1}(x)=\alpha^2 G_{k+1}^{*}(x)G_{k+1}^{*}(x)+\alpha^2g(x)G_{k}^{*}(x)G_{k}^{*}(x).
$$
Thus,
$
(a-b)^2G^{\prime} _{2k+1}(x)-\alpha^2(G^{*}_{k+1}(x))^2=\alpha^2g(x)(G^{*}_k (x))^2.
$
This implies that $$r(x) \text{ divides } \alpha^2 g(x)(G^{*} _k(x))^2.$$
Since $|\alpha| = 1$ or $2$, from the definition of GFPs and Proposition \ref{gcddistance1;2} it is easy to see that
$\gcd(\alpha,g(x))=1$. We know that $\gcd(\alpha, G_n)=1$ for every $n$. So, $\gcd(\alpha, r(x))=1$.
We recall that from Lemma (\ref{gcdlemmas}) part (3), that $\gcd(g(x),G_{k+1}^{*}(x))=1$.
This and  $r(x)\mid G_{k+1}^{*}(x)$ imply that $\gcd(r(x),g(x))=1$. Now it is easy to see that
$r(x)\mid (G^{*} _{k}(x))^2$.  Since $\gcd((a-b)^2,G_{k}^{*}(x))=1$ and $r(x)\mid (a-b)^2$,
we have   $r(x)=1$. This proves that $S(k+1)$ is true. That is, $\gcd((a-b)^2,G_{n}^{*}(x))=1$.

We now prove that $\gcd(G^{\prime}_{m+n+1}(x),G_{n}^{*}(x))=\gcd(G_{m+1}^{*}(x),G_{n}^{*}(x)).$
Proposition \ref{divisity:Hogat:property2} part (1), implies that
$\gcd((a-b)^2G^{\prime}_{m+n+1}(x),G_{n}^{*}(x))$ is equal to
$$\gcd(\alpha^{2}G_{m+1}^{*}(x)G_{n+1}^{*}(x)+\alpha^{2}g(x)G_{m}^{*}(x)G_{n}^{*}(x),G_{n}^{*}(x)).$$ Therefore,
$$\gcd((a-b)^2G^{\prime}_{m+n+1}(x),G_{n}^{*}(x))=\gcd(\alpha^{2} G_{m+1}^{*}(x)G_{n+1}^{*}(x),G_{n}^{*}(x)).$$
Proposition  \ref{gcddistance1;2} and $\gcd (\alpha, G_{n+1})=1$ imply that
$\gcd(\alpha^{2} G_{n+1}^{*}(x),G_{n}^{*}(x))=1$. Therefore, by Proposition \ref{prop2;1} part (2) we have
$$\gcd(G_{m+1}^{*}(x)G_{n+1}^{*}(x),G_{n}^{*}(x))=\gcd(G_{m+1}^{*}(x),G_{n}^{*}(x)).$$
This implies that
$$
\gcd((a-b)^2G^{\prime}_{m+n+1}(x),G_{n}^{*}(x))=\gcd(G_{m+1}^{*}(x),G_{n}^{*}(x)).
$$
This and $\gcd((a-b)^2,G_{n}^{*}(x))=1$ imply that
$$\gcd(G^{\prime}_{m+n+1}(x),G_{n}^{*}(x))=\gcd(G_{m+1}^{*}(x),G_{n}^{*}(x)).$$

Proof of part (2).  From Lemma \ref{gcdlemmas} part (3) it is easy to see that
$$\gcd(G_{m-n+1}^{\prime}(x), G_{n}^{*}(x))=\gcd((g(x))^n G_{m+1-n}^{\prime}(x), G_{n}^{*}(x)).$$
This and Proposition
\ref{divisity:Hogat:property1}  part (2) (after interchanging the roles of $m$ and $n$) imply that
$$\gcd(G_{m-n+1}^{\prime}(x), G_{n}^{*}(x))$$ is equal to
\[\gcd(\alpha G_{m+1}^{\prime}(x)G_{n}^{*}(x)-G_{m+1+n}^{\prime}(x),G_{n}^{*}(x))=\gcd(G_{m+n+1}^{\prime}(x), G_{n}^{*}(x)).\]
The conclusion follows from part (1).

Proof of part (3).  From Lemma \ref{gcdlemmas} part (3) it is easy to see that
\[\gcd(G_{n-m+1}^{\prime}(x), G_{n}^{*}(x)) =\gcd((-g(x))^{m-1} G_{n-(m-1)}^{\prime}(x), G_{n}^{*}(x)).\]
This and Proposition \ref{divisity:Hogat:property1}  part (3) imply that $\gcd(G_{m-n+1}^{\prime}(x), G_{n}^{*}(x))$
is equal to
\[ \gcd(G_{n+m-1}^{\prime}(x)-\alpha G_{m-1}^{\prime}(x)G_{n}^{*}(x),G_{n}^{*}(x))=\gcd(G_{n+(m-2)+1}^{\prime}(x), G_{n}^{*}(x)).\]
The conclusion follows from part (1).
\end{proof}

\begin{theorem}\label{combine:gcd:Lucas:Fibobacci}
Let $G_{n}^{*}(x)$ and $G_{n}^{\prime}(x)$ be equivalent GFPs. If $m$ and $n$ are positive integers and $\gcd(m,n)=d$, then
\[ \gcd(G^{\prime}_{m}(x),G_{n}^{*}(x))=
	\begin{cases}
          G_{d}^{*}(x), 						  & \mbox{ if } E_2(m)>E_2(n); \\
          \gcd(G_{ d }^{*}(x),G_{0}^{*}(x)),	  & \mbox{ otherwise.}
    \end{cases}
\]
\end{theorem}

\begin{proof} Suppose that $E_2(m)>E_2(n)$. We prove this part using cases.

 {\bf Case} $m>n$. Since $m>n$, there is a  positive integer $l$ such that $m=n+l$. Therefore,
$\gcd(G^{\prime}_{m}(x),G_{n}^{*}(x))=\gcd(G^{\prime}_{l-1+n+1}(x),G_{n}^{*}(x))$.
This and Proposition \ref{gcd:fibonaccilucas:samegcdlucaslucas} part (1) imply that
$\gcd(G^{\prime}_{m}(x),G_{n}^{*}(x))=\gcd(G_{l}^{*}(x),G_{n}^{*}(x))$. Since $E_2(m)>E_2(n)$ and $m=n+l$, we have  $E_2(l)=E_2(n)$.
This and Theorem  \ref{second:main:thm} imply that $\gcd(G_{l}^{*}(x),G_{n}^{*}(x))=G_{\gcd(l,n)}^{*}(x)$.
 From Lemma \ref{Dic1} it is easy to see that $\gcd(l,n)=\gcd(m,n)$. Thus, $\gcd(G_{l}^{*}(x),G_{n}^{*}(x))=G_{\gcd(m,n)}^{*}(x)$. Therefore, we have  $\gcd(G^{\prime}_{m}(x),G_{n}^{*}(x))=G_{\gcd(m,n)}^{*}(x)$.

 {\bf Case} $m<n$. The proof of this case is similar to the proof of Case $m>n$. It is enough to replace
 $m$ by $n-(l+1)$ in $\gcd(G^{\prime}_{m}(x),G_{n}^{*}(x))$, and then use Proposition \ref{gcd:fibonaccilucas:samegcdlucaslucas} part (3).

We now suppose that $E_2(m)\le E_2(n)$. Again we argue using cases.

 {\bf Case} $m>n$. So, there is a  positive integer $l$ such that $m= l+n$.  Therefore, by
 Proposition \ref{gcd:fibonaccilucas:samegcdlucaslucas} part (1) we have
 \[\gcd(G^{\prime}_{m}(x),G_{n}^{*}(x))=\gcd(G^{\prime}_{n+(l-1)+1}(x),G_{n}^{*}(x))=\gcd(G_{l}^{*}(x),G_{n}^{*}(x)).\]
 Note that if $m=n+l$ and $E_2(m)\le E_2(n)$, then there are integers $k_1$, $k_2$, $q_1$, and $q_2$ with $k_1\leq k_2$
 such that $m=2^{k_1} q_1$ and $n=2^{k_2} q_2$. Since $m=n+l$, we see that $E_2(l)\neq  E_2(n)$.
 This and Theorem  \ref{second:main:thm} imply that $$\gcd(G_{l}^{*}(x),G_{n}^{*}(x))=\gcd(G_{0}^{*}(x),G_{\gcd(n,l)}^{*}(x)).$$
 Thus, $$\gcd(G^{\prime}_{m}(x),G_{n}^{*}(x))=\gcd(G_{0}^{*}(x),G_{\gcd(m,n)}^{*}(x)).$$

 {\bf Case} $m<n$. The proof of this case is similar to the proof of Case $m>n$. It is enough to replace $m$ by
 $n-(l+1)+1$ in $\gcd(G^{\prime}_{m}(x),G_{n}^{*}(x))$, and then use Proposition \ref{gcd:fibonaccilucas:samegcdlucaslucas} part (3).

 {\bf Case} $m=n$. Since $n=(2n-1)-n+1$, taking $m=2n-1$ in Proposition \ref{gcd:fibonaccilucas:samegcdlucaslucas} part (2)
 and using Theorem \ref{second:main:thm} we obtain that $\gcd(G'_{n}(x),G^{*}_{n}(x))$ is equal to
$$
\gcd(G^{\prime}_{(2n-1)-n+1}(x),G^{*} _{n}(x))=\gcd(G^{*} _{2n}(x),G^{*} _n(x))=\gcd(G_{0}^{*}(x),G_{n}^{*}(x)).
$$
This completes the proof.
\end{proof}

In \cite[Theorem 3.4]{hoggatt} is proved that Fibonacci polynomials,
Chebyshev polynomials of second kind, Morgan-Voyce polynomials, and Schechter polynomials satisfy the strong divisibility property.
In \cite[Theorem 7]{HoggattLong} is proved that GFP's of Fibonacci type satisfy the strong divisibility property.
Theorem \ref{gcd:property:fibonacci} proves a necessary and sufficient condition for the polynomials in a
generalized Fibonacci polynomial sequence to satisfy the strong divisibility property. Norfleet \cite{norfleet} also
proved the same strong divisibility property for GFPs of Fibonacci type.

\begin{theorem} \label{gcd:property:fibonacci}
	Let $\{ G_{k}(x)\}$ be a GFP sequence that is either Fibonacci type or Lucas type. For any two positive integers $m$ and $n$ it holds that
	$\{ G_{k}(x)\}$ is a sequence of GFPs of Fibonacci type if and only if  $\gcd(G_{m}(x),G_{n}(x))=G_{\gcd(m,n)}(x)$.
\end{theorem}

\begin{proof}  Let $\{ G_{n}^{\prime}(x)\}$ be a GFP sequence of
Fibonacci type, we now show that
$\gcd(G_{m}^{\prime}(x),G_{n}^{\prime}(x))$ divides $G_{\gcd(m,n)}^{\prime}(x)$ for $m>0$, $n>0$
and vice versa.

If $G_{n}^{\prime}(x)$ is of Fibonacci type, by Proposition \ref{divisity:property:fibonacci} it is clear that
$$G_{\gcd(m,n)}^{\prime}(x) \mid \gcd(G_{m}^{\prime}(x),G_{n}^{\prime}(x)).$$ Next we show that
$\gcd(G_{m}^{\prime}(x),G_{n}^{\prime}(x))$ divides $G_{\gcd(m,n)}^{\prime}(x)$.

Let $k=\gcd(m,n)$ and assume without the loss of generality that $k \ne n$ and $k \ne m$.
The B\'{e}zout identity implies that there are two positive integers $r$ and $s$ such that $k=rm-sn$.
So, $rm=k+sn$ and $G_{rm}^{\prime}(x) =G_{k+sn}^{\prime}(x)$.
This,  Proposition \ref{divisity:Hogat:property1} part (1), and the fact that $k+sn=(k+(sn-1))+1$ imply that
\[ G_{rm}^{\prime}(x)= G_{k+1}^{\prime}(x)G_{s'n}^{\prime}(x)+g(x)G_{k}^{\prime}(x)G_{sn-1}^{\prime}(x). \]

We note that by Proposition \ref{divisity:property:fibonacci}, $G_{m}^{\prime}(x)\mid G_{rm}^{\prime}(x)$
and $G_{n}^{\prime}(x)\mid G_{sn}^{\prime}(x)$. Since both $$\gcd(G_{m}^{\prime}(x),G_{n}^{\prime}(x)) \mid G_{m}^{\prime}(x)\text{ and } \gcd(G_{m}^{\prime}(x),G_{n}^{\prime}(x)) \mid G_{n}^{\prime}(x)$$ hold
and also both $G_{m}^{\prime}(x)\mid G_{rm}^{\prime}(x)$ and $G_{n}^{\prime}(x)\mid G_{s'n}^{\prime}(x)$ hold, we have that
 $\gcd(G_{m}^{\prime}(x),G_{n}^{\prime}(x))$ divides both $G_{rm}^{\prime}(x)$ and $G_{s'n}^{\prime}(x)$. This
 together with Lemma \ref{gcdlemmas} part (3) and the fact that $\gcd(G_{m}^{\prime}(x),G_{n}^{\prime}(x))$
 does not divide $G_{s'n-1}^{\prime}(x)$ imply that $\gcd(G_{m}^{\prime}(x),G_{n}^{\prime}(x))$
 divides $G_{k}^{\prime}(x)$.

 Conversely, suppose that $\{ G_{n}(x)\}$ is a GFP sequence such that the strong divisibility property holds
 or $\gcd(G_{m}(x),G_{n}(x))=G_{\gcd(m,n)}(x)$ for any two positive integers $m$ and $n$. We now show that both $G_{m}(x)$ and
 $G_{n}(x)$ are GFPs of the Fibonacci type.
  We prove this using the method of contradiction.

 If $G_{m}(x)$ and $G_{n}(x)$ are in $\{ G_{n}(x)\}$ such that they are both GFPs of the Lucas type, then by
 Theorem \ref{second:main:thm} we obtain a contradiction.
 This completes the proof.
\end{proof}

\section{The gcd properties of familiar GFPs and questions}

In this section we formulate a general question and present three tables which are corollaries of the main
results in Section \ref{gcd:characterization}. These tables give us the strong divisibility property of the
familiar polynomials which satisfy the Binet formulas  (\ref{bineformulados}) and (\ref{bineformulauno}).
Table \ref{corollary_Fibonacci} gives the greatest common divisors for Fibonacci polynomials, Pell polynomials, Fermat
polynomials, Jacobsthal polynomials, Chebyshev polynomials of the second kind, and one type of Morgan-Voyce ($B_n$) polynomials.
Table \ref{corollary_lucas} gives the strong divisibility property of the Lucas polynomials, Pell-Lucas polynomials,
Fermat-Lucas polynomials, Jacobsthal-Lucas polynomials, Chebyshev polynomials of the first kind, and Morgan-Voyce ($C_n$) polynomials.
Table \ref{corollary_lucas_2} gives the $\gcd$ of a polynomial of Lucas type and its equivalent polynomial of Fibonacci type.

We note here that in the case of Table \ref{corollary_lucas}, the strong divisibility property is partially satisfied since
it only holds when the largest power of 2 that divides $m$ and the largest power of 2 that divides $n$ are equal.
(That is, $E_{2}(m)= E_{2}(n)$.) Similarly the strong divisibility property only holds in Table \ref{corollary_lucas_2}
when $E_{2}(n)<E_{2}(m)$.

For simplicity we present the polynomials in Tables \ref{corollary_Fibonacci}, \ref{corollary_lucas} and \ref{corollary_lucas_2} without the variable  $x$.
\begin{table} [!ht]
\begin{center}\scalebox{1}{
\begin{tabular}{|l|l|l|} \hline
 Polynomial      	 	 	& The Fibonacci $\gcd$ property                   \\ \hline \hline
 Fibonacci            	 	& $\gcd(F_{m}, F_{n})=F_{\gcd(m,n)}$              \\
 Pell		         	 	& $\gcd(P_{m}, P_{n})=P_{\gcd(m,n)}$	          \\
 Fermat		             	& $\gcd(\Phi_{m}, \Phi_{n})= \Phi_{\gcd(m,n)}$    \\
 Chebyshev the second kind      & $\gcd(U_{m}, U_{n})\,=U_{\gcd(m,n)}$            \\
 Jacobsthal		            & $\gcd(J_{m}, J_{n})\,\,=\,J_{\gcd(m,n)}$        \\
 Morgan-Voyce            	&$\gcd(B_{m}, B_{n})=B_{\gcd(m,n)}$                \\
 \hline
\end{tabular}}
\end{center}
\caption{Strong divisibility property of polynomials of Fibonacci type.} \label{corollary_Fibonacci}
\end{table}

\begin{table} [!ht]
	\begin{center}\scalebox{1}{
			\begin{tabular}{|l|l|l|} \hline
				Polynomial      	 	& Case 1: $E_{2}(m)= E_{2}(n)$                                      & Case 2: $E_{2}(m)\ne E_{2}(n)$               \\ \hline \hline
				Lucas               	& $\gcd(D_{m}, D_{n})=D_{\gcd(m,n)}$                                & $\gcd(D_{m}, D_{n})=\,1$                     \\
				Pell-Lucas-prime        & $\gcd(Q_{m}^{'}, Q_{n}^{'})=Q_{\gcd(m,n)}^{'}$                    & $\gcd(Q_{m}^{'}, Q_{n}^{'})=\,1$             \\
				Fermat-Lucas         	& $\gcd(\vartheta_{m}, \vartheta_{n})\,\,=\vartheta_{\gcd(m,n)}$  	& $\gcd(\vartheta_{m}, \vartheta_{n})\,\,=\,1$ \\
				Chebyshev the first kind    & $\gcd(T_{m}, T_{n})\,\,=T_{\gcd(m,n)}$                            & $\gcd(T_{m}, T_{n})\,\,=\,1$                 \\
				Jacobsthal-Lucas        & $\gcd(Q_{m}, Q_{n})=Q_{\gcd(m,n)}$                                & $\gcd(Q_{m}, Q_{n})=\,1$                     \\
				Morgan-Voyce         	& $\gcd(C_{m}, C_{n})\,= C_{\gcd(m,n)}$                             & $\gcd(C_{m}, C_{n})\,=\,1$                   \\
				\hline
		\end{tabular}}
	\end{center}
	\caption{Strong divisibility property (partially) of polynomials of Lucas type.} \label{corollary_lucas}
\end{table}

\begin{table} [!ht]
	\begin{center}\scalebox{1}{
			\begin{tabular}{|l|l|l|} \hline
				Polynomials     	 	  	&  $E_{2}(n)< E_{2}(m)$     					& Otherwise                         \\ \hline \hline
				Fibonacci, Lucas           	& $\gcd(F_{m},D_{n})=D_{d}$                     & $\gcd(F_{m},D_{n})=\,1$           \\
				Pell, Pell-Lucas-prime      & $\gcd(P_{m},Q_{n}^{'})=Q_{d}^{'}$             & $\gcd(P_{m},Q_{n}^{'})=\,1$       \\
				Fermat, Fermat-Lucas       	& $\gcd(\Phi_{m},\vartheta_{n})=\vartheta_{d}$  & $\gcd(\Phi_{m},\vartheta_{n})=\,1$\\
				Chebyshev both kinds 		& $\gcd(U_{m},T_{n})=T_{d}$                     & $\gcd(U_{m},T_{n})=\,1$           \\
				Jacobstal, Jacobsthal-Lucas	& $\gcd(J_{m},j_{n},)=j_{d}$                    & $\gcd(J_{m},j_{n})=\;1$           \\
				Morgan-Voyce both types    	& $\gcd(B_{m},C_{n})= C_{d}$                    & $\gcd(B_{m},C_{n})=\,1$           \\
				\hline
		\end{tabular}}
	\end{center}
	\caption{Strong divisibility property (partially) of polynomials of Lucas type and their equivalents, where $d=\gcd(m,n)$.} \label{corollary_lucas_2}
\end{table}

\subsection{Questions}
\begin{enumerate}
  \item  Let $\{G_{n}^{*}(x)\}$ and $\{ S_{n}(x)\}$ be generalized Fibonacci
polynomial sequences of Lucas type and Fibonacci type, respectively. If $S_{n}(x) $ is not the equivalent of $G_{n}^{*}(x)$, what is the
$\gcd(G_{k}^{*}(x),S_{m}(x))$? We believe that the answer is: $1$ or $x$.

  \item If $\{G_{n}(x)\}$ and $\{S_{n}(x)\}$ are two different generalized Fibonacci polynomial sequences of the same type, then do they
  satisfy the strong divisibility property?

  \item ({\bf Conjecture.}) The GFPs $T_{n}$ and $S_{m}$ satisfy the strong divisibility property if and only if $T_{n}$ and $S_{m}$
  are both of Fibonacci type and they belong to the same generalized Fibonacci polynomial sequence.
Theorems \ref{combine:gcd:Lucas:Fibobacci}, \ref{gcd:property:fibonacci}, and \cite[Lemma 4]{bliss} suggest that the conjecture is true.

\item Let $\mathcal{R}$ be a set of recursive functions. If
$\mathcal{F}  :  \mathbb{N}  \to \mathcal{R}$,  $\mathcal{G} :  \mathcal{R} \times \mathcal{R} \to \mathcal{R}$, and
$g: \mathbb{N} \times \mathbb{N} \to \mathbb{N}$. Under what conditions
$\mathcal{G} \circ  (\mathcal{F} \times \mathcal{F} )= \mathcal{F} \circ g$ for all $\mathcal{F}\in \mathcal{R}$ and a fixed $g$?

\end{enumerate}

\section{Acknowledgement}

We would like to thank an anonymous referee for
suggesting that we refer to papers \cite{bliss, nowicki} which helped us improve the paper.

\bigskip
\hrule
\bigskip

\noindent  MSC 2010:
Primary 11B39; Secondary 11B83.

\noindent \emph{Keywords: }
Greatest common divisor, strong divisibility property,  generalized Fibonacci polynomial, Fibonacci polynomial, Lucas polynomial.

\end{document}